\documentclass[11pt]{article}

\usepackage[T1]{fontenc}
\usepackage[utf8]{inputenc}
\usepackage{lmodern}
\usepackage[margin=1.08in]{geometry}
\usepackage{amsmath,amssymb,amsthm,mathtools,mathrsfs}
\usepackage{microtype}
\usepackage[colorlinks=true,linkcolor=blue,citecolor=blue,urlcolor=blue,
  pdftitle={Eventually Tur\'an good II: Vertex-Linear Thresholds and the Cluster-Expansion Method},
  pdfsubject={Vertex-linear thresholds for Tur\'an-goodness and the cluster-expansion method},
  pdfkeywords={generalized Tur\'an problem; Tur\'an-good graph; cluster expansion; polymer model; chromatic polynomial},
  pdfauthor={Xiamiao Zhao, Liying Kang, and Yuanpei Wang}]{hyperref}

\allowdisplaybreaks

\numberwithin{equation}{section}

\newtheorem{theorem}{Theorem}[section]
\newtheorem{lemma}[theorem]{Lemma}

\newtheorem{corollary}[theorem]{Corollary}

\theoremstyle{definition}

\theoremstyle{remark}

\newcommand{\inj}{\operatorname{inj}}
\newcommand{\Inj}{\operatorname{Inj}}
\newcommand{\Aut}{\operatorname{Aut}}
\newcommand{\homc}{\operatorname{hom}}
\newcommand{\ex}{\operatorname{ex}}
\newcommand{\E}{\mathbb E}
\newcommand{\Prb}{\mathbb P}
\newcommand{\cP}{\mathcal P}
\newcommand{\cT}{\mathcal T}
\newcommand{\cC}{\mathcal C}
\newcommand{\1}{\mathbf 1}
\newcommand{\dd}{\,\mathrm d}
\newcommand{\dotcup}{\mathbin{\dot\cup}}

\title{%
  \fontsize{18}{22}\selectfont
  \bfseries
   A Vertex-Linear Threshold for Eventually Tur\'an good
  and the Cluster Method%
}
\author{%
  Xiamiao Zhao\textsuperscript{*}
  \qquad
  Liying Kang\textsuperscript{\(\dagger\)}
  \qquad
  Yuanpei Wang\textsuperscript{\(\ddagger\)}%
}
\date{}

\begin{document}
\maketitle
\begingroup
\renewcommand{\thefootnote}{\fnsymbol{footnote}}
\footnotetext[1]{Department of Mathematical Sciences, Tsinghua University, Beijing 100084, P.R. China. Supported by the China Scholarship Council (No. 202506210250) and the National Natural Science Foundation of China (Grant No. 12571372). Email: \href{mailto:zxm23@mails.tsinghua.edu.cn}{zxm23@mails.tsinghua.edu.cn}.}
\footnotetext[2]{Department of Mathematics, Shanghai University, Shanghai 200444, P.R. China, and Newtouch Center for Mathematics of Shanghai University, Shanghai 200444, P.R. China. Supported by the National Natural Science Foundation of China (Grant Nos. 12571375 and 12331012). Email: \href{mailto:lykang@shu.edu.cn}{lykang@shu.edu.cn}.}
\footnotetext[3]{Corresponding author. Department of Mathematics, Shanghai University, Shanghai 200444, P.R. China. Supported by the China Scholarship Council (No. 202506890067). Email: \href{mailto:boyuan@shu.edu.cn}{boyuan@shu.edu.cn}.}
\endgroup

\begin{abstract}
A graph $H$ is $K_{r+1}$-Tur\'an-good if, for every sufficiently large $n$, the Tur\'an graph $T_r(n)$ maximizes the number of copies of $H$ among all $n$-vertex $K_{r+1}$-free graphs, and it is strictly $K_{r+1}$-Tur\'an-good when this extremal graph is unique. Morrison, Nir, Norin, Rz\k{a}\.zewski and Wesolek proved that for every graph $H$ with at least one edge, when $r\ge 300v(H)^9$, $H$ is $K_{r+1}$-Tur\'an-good.
They asked whether the above bound could be reduced to quadratic order in $v(H)$.

In this paper, we answer their question affirmatively, and give a stronger result.
 We prove that there is an absolute constant $C>0$ such that every graph $H$ with at least one edge is strictly $K_{r+1}$-Tur\'an-good whenever $r\ge Cv(H)$.
Our proof uses a substantially different cluster method based on polymer models. Besides proving the vertex-linear Tur\'an-good threshold, this method appears to have further applications. As one illustration, we prove that for every graph $H$ with at least one edge, the normalized chromatic polynomial $P_H(x)/x^{v(H)}$ is strictly increasing for every real $x\ge C\Delta(H)$, which strengthens the previously best result where $x\ge C\Delta(H)^{3/2}$. 
This proves a conjecture of Fadnavis.
\end{abstract}


\section{Introduction}

\subsection{Tur\'an-goodness and the cluster method}

For fixed graphs $H$ and $F$, the generalized Tur\'an number is
\[
\ex(n,H,F):=\max\{N(H,G): |V(G)|=n,\ G\text{ is }F\text{-free}\},
\]
where $N(H,G)$ denotes the number of unlabeled copies of $H$ in $G$. The generalized Tur\'an problem was introduced by Alon and Shikhelman \cite{AlonShikhelman}. For further results, we refer to the survey of Gerbner and Palmer \cite{GerbnerPalmerSurvey}. The case $H=K_2$ is the classical Tur\'an problem.

Let $T_r(n)$ denote the Tur\'an graph, the complete $r$-partite graph on $n$ vertices whose part sizes differ by at most one. Following Gerbner and Palmer \cite{GerbnerPalmerCounting,GerbnerPalmerExact}, a graph $H$ is $K_{r+1}$-Tur\'an-good if
\[
\ex(n,H,K_{r+1})=N(H,T_r(n))
\]
for every sufficiently large $n$. If $T_r(n)$ is the unique extremal graph for every sufficiently large $n$, then $H$ is strictly $K_{r+1}$-Tur\'an-good. Zykov's theorem \cite{Zykov} shows that every clique $K_t$ with $t\le r$ is strictly $K_{r+1}$-Tur\'an-good. Further exact and stability results for several graph classes were obtained by Ma and Qiu \cite{MaQiu}.

An injective homomorphism from $H$ to $G$ is an injective map $\varphi:V(H)\to V(G)$ such that $uv\in E(H)$ implies $\varphi(u)\varphi(v)\in E(G)$. Let $\Inj(H,G)$ be the set of all such maps, and let
\[
\inj(H,G):=|\Inj(H,G)|.
\]
Since
\[
\inj(H,G)=|\Aut(H)|N(H,G),
\]
maximizing $\inj(H,G)$ is equivalent to maximizing $N(H,G)$. For a nonnegative integer $a$ and an integer $b\ge0$, define $(a)_b:=a(a-1)\cdots(a-b+1)$.

Gerbner and Palmer \cite{GerbnerPalmerCounting,GerbnerPalmerExact} conjectured that every graph is eventually Tur\'an-good. Morrison, Nir, Norin, Rz\k{a}\.zewski and Wesolek \cite{MorrisonEtAl} proved the following theorem.

\begin{theorem}[Morrison, Nir, Norin, Rz\k{a}\.zewski and Wesolek \cite{MorrisonEtAl}]\label{thm:MNNRW}
Let $H$ be a graph and let $r\ge300v(H)^9$. Then $H$ is $K_{r+1}$-Tur\'an-good.
\end{theorem}

They observed that improving this bound even to quadratic order in $v(H)$ would likely require additional ideas. Very recently, Wang Kang and Zhao \cite{PartI} proved the following explicit edge-linear result.

\begin{theorem}[Wang, Kang and Zhao \cite{PartI}]\label{thm:edge-linear}
For every graph $H$ with at least one edge and every integer $r\ge168e(H)$, the graph $H$ is strictly $K_{r+1}$-Tur\'an-good and $K_{r+1}$-Tur\'an-stable.
\end{theorem}

Every edgeless graph is trivially $K_{r+1}$-Tur\'an-good, but not strictly so. Isolated vertices may be removed without changing the relevant extremal conclusions; the precise identity is recorded in Subsection~\ref{subsec:notation-organization}.

Theorem~\ref{thm:edge-linear} gives a quadratic sufficient condition in terms of $v(H)$, while our first main result improves this to a universal linear bound.

\begin{theorem}\label{thm:main}
There is an absolute constant $C>0$ such that, for every graph $H$ with at least one edge and every integer $r\ge Cv(H)$, the graph $H$ is strictly $K_{r+1}$-Tur\'an-good.
\end{theorem}

Theorems~\ref{thm:edge-linear} and~\ref{thm:main} are complementary: when $H$ is sparse, Theorem~\ref{thm:edge-linear} may give a better range of $r$, whereas for dense $H$, Theorem~\ref{thm:main} gives the stronger bound.

The proof in the present paper is different from  \cite{PartI}. We use the cluster method for abstract polymer models, developed in statistical mechanics by Koteck\'y and Preiss \cite{KoteckyPreiss} and Fern\'andez and Procacci \cite{FernandezProcacci}, and used for graph polynomials by Scott and Sokal \cite{ScottSokal}. 
Cluster method have also been used in enumerative Tur\'an-type problems \cite{JenssenPerkinsPotukuchi}. The method appears to have further potential, as illustrated by the chromatic-polynomial problem below.

\subsection{Chromatic-polynomial monotonicity}

For a graph $H$, let $P_H(x)$ denote its chromatic polynomial, the unique polynomial such that $P_H(q)$ is the number of proper $q$-colorings of $H$ for every positive integer $q$. Chromatic polynomials were introduced by Birkhoff \cite{Birkhoff} in connection with the four-color problem. If $h=v(H)$, then
$
p_q(H):=\frac{P_H(q)}{q^h}
$
is the probability that a uniformly random $q$-coloring of $V(H)$ is proper.

Bartels and Welsh \cite{BartelsWelsh} conjectured that every $h$-vertex graph satisfies
\begin{equation*}
 \frac{P_H(h)}{h^h}\ge \frac{P_H(h-1)}{(h-1)^h}
\end{equation*}

Dong \cite{Dong} proved the stronger one-step inequality
$
\frac{P_H(x)}{x^h}\ge \frac{P_H(x-1)}{(x-1)^h}
$
 holds for every real $x\ge h$. Fadnavis \cite{Fadnavis} replaced the dependence on $h$ by the sufficient condition $x>36\Delta(H)^{3/2}$ and conjectured that a condition of the form $x>C\Delta(H)$ should suffice for some absolute constant $C$. Ning and Yang \cite{NingYang} recently improved the constant $36$ to $10$, while retaining the exponent $3/2$.

Our second main theorem resolves the conjecture of Fadnavis and proves a stronger continuous monotonicity statement.

\begin{theorem}\label{thm:chromatic}
There are absolute constants $C,K>0$ with $C\ge100K$ such that, for every graph $H$ with at least one edge and every real $x\ge C\Delta(H)$,
\[
\left|
\frac{\dd}{\dd x}\log\!\left(\frac{P_H(x)}{x^{v(H)}}\right)-\frac{e(H)}{x^2}
\right|
\le K\frac{e(H)\Delta(H)}{x^3}.
\]
In particular, $P_H(x)/x^{v(H)}$ is strictly increasing on $[C\Delta(H),\infty)$.
\end{theorem}

Theorem~\ref{thm:chromatic} is stronger than the consecutive-value comparison conjectured by Fadnavis. Integrating its logarithmic derivative over two adjacent coloring levels also gives
\[
\log\frac{p_{r-1}(H)}{p_r(H)}
=-\frac{e(H)}{r(r-1)}+O\!\left(\frac{e(H)\Delta(H)}{r^3}\right),
\]
which is used in the proof of Theorem~\ref{thm:main}.

\subsection{Notation and organization}\label{subsec:notation-organization}

Throughout the paper, all graphs are finite and simple. For a graph $G$, let $V(G)$ and $E(G)$ denote its vertex and edge sets, and let $v(G):=|V(G)|$ and $e(G):=|E(G)|$. For $u\in V(G)$, let $N_G(u)$ be its neighborhood and let $d_G(u):=|N_G(u)|$. Let $\delta(G)$, $\Delta(G)$, and $\chi(G)$ denote the minimum degree, maximum degree, and chromatic number of $G$, respectively.

If $H=H_0\dotcup tK_1$, where $H_0$ has no isolated vertices, then
\begin{equation}\label{eq:isolated}
\inj(H,G)=(n-v(H_0))_t\inj(H_0,G)
\end{equation}
for every $n$-vertex graph $G$. The factor is independent of $G$.

The proof of Theorem~\ref{thm:main} combines these estimates with induction and F\"uredi's multipartite reduction. After establishing a high minimum-degree condition, switching and balancing arguments show that an extremal graph must be isomorphic to $T_r(n)$. Section~2 develops the polymer estimates, Section~3 proves the multipartite balancing estimates, Section~4 proves Theorem~\ref{thm:chromatic}, and Section~5 proves Theorem~\ref{thm:main}.

\section{Polymer estimates}\label{sec:polymer}

All asymptotic notation in the extremal part of the proof refers to $n\to\infty$ with $H$ and $r$ fixed. Constants denoted by $C,c$ without subscripts are absolute and may change from line to line. Constants in $O_{H,r}(\cdot)$ may depend on $H$ and $r$. For graphs $J$ and $F$, let $\homc(J,F)$ denote the number of all graph homomorphisms from $J$ to $F$, and let $\cT(J)$ denote the set of spanning trees of $J$.

For every graph $J$ on $j$ vertices and every $N$-vertex graph $F$,
\begin{equation}\label{eq:collision}
0\le \homc(J,F)-\inj(J,F)\le \binom{j}{2}N^{j-1}.
\end{equation}
Indeed, every noninjective map identifies at least one pair of vertices of $J$.

By \eqref{eq:isolated}, isolated vertices can be removed without changing the extremal graphs. We shall therefore assume, when proving Theorem~\ref{thm:main}, that $H$ has no isolated vertices.

We use the standard polymer-model terminology of Koteck\'y--Preiss \cite{KoteckyPreiss} and Fern\'andez--Procacci \cite{FernandezProcacci}; every notion needed in the proofs is defined explicitly below.

\subsection{Abstract polymer estimates}

Let $X$ be a finite ground set and let $\cP\subseteq 2^X\setminus\{\varnothing\}$. The members of $\cP$ are called \emph{polymers}. Two polymers $\gamma,\eta$ are compatible when $\gamma\cap\eta=\varnothing$ and incompatible otherwise. In particular, every polymer is incompatible with itself. An \emph{activity} is a complex weight $z_\gamma$ assigned to a polymer $\gamma$; an activity system is $z=(z_\gamma)_{\gamma\in\cP}$. Define
\[
\Xi(z)=
\sum_{\substack{\Gamma\subseteq\cP\\
                  \Gamma\text{ pairwise compatible}}}
\prod_{\gamma\in\Gamma}z_\gamma.
\]
Thus $\Gamma$ is a family of pairwise disjoint polymers, and the empty family contributes $1$ to the sum in $\Xi(z)$. We use $|z_\gamma|$ for the complex modulus of an activity and $|\gamma|$ for the cardinality of the set $\gamma$.

For an ordered tuple $(\gamma_1,\ldots,\gamma_k)$, its simple incompatibility graph $I(\gamma_1,\ldots,\gamma_k)$ has vertex set $[k]$, with $ij$ an edge precisely when $\gamma_i\cap\gamma_j\ne\varnothing$.
We call the ordered tuple $(\gamma_1,\ldots,\gamma_k)$ a
\emph{cluster} if its incompatibility graph
$I(\gamma_1,\ldots,\gamma_k)$ is connected. Thus the polymers in a
cluster need not be pairwise incompatible; it is enough that every
two of them can be joined by a chain of incompatibilities. In the
present set-polymer model, this means a chain of nonempty
intersections. Repetitions are allowed, since the expansion below is
over ordered tuples in $\mathcal P^k$. This agrees with the standard
terminology for abstract polymer models
\cite{FernandezProcacci}.
We define 
\[
\phi^{\mathrm T}(\gamma_1,\ldots,\gamma_k)
=
\sum_{\substack{A\subseteq E(I(\gamma_1,\ldots,\gamma_k))\\
                  ([k],A)\text{ connected}}}
(-1)^{|A|}.
\]
Here $([k],A)$ retains the full vertex set $[k]$. For $k=1$ it is connected and the empty edge set gives $\phi^{\mathrm T}(\gamma_1)=1$. If the incompatibility graph is disconnected, the displayed sum is empty and the coefficient is zero.

Let
\[
f(\gamma,\eta)=
\begin{cases}
-1,&\gamma\cap\eta\ne\varnothing,\\
0,&\gamma\cap\eta=\varnothing.
\end{cases}
\]
For every ordered tuple $(\gamma_1,\ldots,\gamma_k)$,
\[
\prod_{1\le i<j\le k}\bigl(1+f(\gamma_i,\gamma_j)\bigr)
=
\begin{cases}
1,&(\gamma_1,\ldots,\gamma_k)\text{ are pairwise compatible},\\
0,&\text{otherwise}.
\end{cases}
\]
For every activity system $w=(w_\gamma)_{\gamma\in\cP}$,
writing every compatible family in all possible orders therefore gives the formal identity
\begin{equation}\label{eq:Xi-ordered}
\Xi(w)=1+\sum_{k\ge1}\frac1{k!}
\sum_{\gamma_1,\ldots,\gamma_k\in\cP}
\left(\prod_{i=1}^k w_{\gamma_i}\right)
\prod_{1\le i<j\le k}\bigl(1+f(\gamma_i,\gamma_j)\bigr).
\end{equation}
Repeated polymers contribute zero because a nonempty polymer intersects itself. Expanding the last product gives
\begin{equation}\label{eq:graph-expansion}
\prod_{i<j}\bigl(1+f(\gamma_i,\gamma_j)\bigr)
=
\sum_{G\text{ on }[k]}\left(\prod_{ij\in E(G)}f(\gamma_i,\gamma_j)\right).
\end{equation}

Let $\cC_k$ be the set of connected graphs on $[k]$ and define
\begin{equation}\label{eq:Cw}
C(w)=\sum_{k\ge1}\frac1{k!}
\sum_{\gamma_1,\ldots,\gamma_k\in\cP}
\left(
\sum_{G\in\cC_k}\prod_{ij\in E(G)}f(\gamma_i,\gamma_j)
\right)
\prod_{i=1}^k w_{\gamma_i}.
\end{equation}
To see that $\Xi(w)=\exp C(w)$, fix a graph $G$ on $[k]$ and let $\pi(G)$ be the partition of $[k]$ into the vertex sets of its connected components. Then
\[
\prod_{ij\in E(G)}f(\gamma_i,\gamma_j)
=
\prod_{B\in\pi(G)}
\prod_{ij\in E(G[B])}f(\gamma_i,\gamma_j).
\]
Conversely, a set partition $\pi$ of $[k]$ together with one connected graph on every block $B\in\pi$ determines a unique graph whose connected components have vertex sets $B$. Hence, for every fixed ordered tuple,
\[
\sum_{G\text{ on }[k]}\left(\prod_{ij\in E(G)}f(\gamma_i,\gamma_j)\right)
=
\sum_{\pi}
\prod_{B\in\pi}
\left(
\sum_{G_B\in\cC_B}\prod_{ij\in E(G_B)}f(\gamma_i,\gamma_j)
\right).
\]
The first sum is over all set partitions $\pi$ of $[k]$.
The expansion of $\exp C(w)=\sum_{s\ge0}C(w)^s/s!$ produces exactly the right-hand side: the factor $1/s!$ forgets the order of the $s$ connected components, while their vertex sets form the partition $\pi$. Comparing with \eqref{eq:Xi-ordered} and \eqref{eq:graph-expansion} yields $\Xi(w)=\exp C(w)$.

For a fixed tuple, a connected graph contributes to the inner sum in \eqref{eq:Cw} only if every one of its edges belongs to $I(\gamma_1,\ldots,\gamma_k)$. In that case every edge factor is $-1$, so the connected-graph sum is
\[
\sum_{\substack{A\subseteq E(I(\gamma_1,\ldots,\gamma_k))\\
                  ([k],A)\text{ connected}}}
(-1)^{|A|}
=
\phi^{\mathrm T}(\gamma_1,\ldots,\gamma_k).
\]
Therefore
\begin{equation}\label{eq:cluster-log}
\log\Xi(w)
=
\sum_{k\ge1}\frac1{k!}
\sum_{\gamma_1,\ldots,\gamma_k\in\cP}
\phi^{\mathrm T}(\gamma_1,\ldots,\gamma_k)
\prod_{i=1}^k w_{\gamma_i}.
\end{equation}
Since $\phi^{\mathrm T}(\gamma_1,\ldots,\gamma_k)=0$ whenever
$I(\gamma_1,\ldots,\gamma_k)$ is disconnected,
\eqref{eq:cluster-log} is precisely a sum over clusters. This is the
sense in which passing to the logarithm retains only connected
contributions.

For nonnegative radii $\bar z=(\bar z_\gamma)_{\gamma\in\cP}$, define the closed polydisc
\[
D(\bar z)=\{w=(w_\gamma)_{\gamma\in\cP}: |w_\gamma|\le\bar z_\gamma\text{ for every }\gamma\}.
\]
The following standard cluster lemma is the external input used below.

\begin{lemma}[Koteck\'y--Preiss \cite{KoteckyPreiss}, Fern\'andez--Procacci \cite{FernandezProcacci}]\label{lem:KP}
Let $a:\cP\to[0,\infty)$ and suppose that
\[
\sum_{\eta:\,\eta\cap\gamma\ne\varnothing}\bar z_\eta e^{a(\eta)}
\le a(\gamma)
\qquad(\gamma\in\cP).
\]
Then $\Xi(w)\ne0$ throughout $D(\bar z)$. The series \eqref{eq:cluster-log} converges absolutely there and may be differentiated term by term. Moreover, for every root polymer $\gamma$,
\begin{equation}\label{eq:rooted-cluster}
\sum_{k\ge1}\frac1{(k-1)!}
\sum_{\gamma_2,\ldots,\gamma_k\in\cP}
\left|\phi^{\mathrm T}(\gamma,\gamma_2,\ldots,\gamma_k)\right|
\prod_{j=2}^k\bar z_{\gamma_j}
\le e^{a(\gamma)}.
\end{equation}
For $k=1$, the inner product and the inner sum each consist of the single term $1$.
\end{lemma}

\begin{lemma}\label{lem:polymer-perturbation}
Let $z,z'$ be two activity systems for the same polymer model on subsets of $X$. Define
\[
\bar z_\gamma=\max\{|z_\gamma|,|z'_\gamma|\},
\qquad
\delta=\max_{x\in X}\sum_{\gamma\ni x}\bar z_\gamma e^{|\gamma|}.
\]
If $\delta\le1/4$, then $\Xi(z)$ and $\Xi(z')$ are nonzero and
\begin{equation}\label{eq:polymer-perturbation}
\left|
\log\Xi(z)-\log\Xi(z')-\sum_{\gamma\in\cP}(z_\gamma-z'_\gamma)
\right|
\le
\sum_{\gamma\in\cP}|z_\gamma-z'_\gamma|
\bigl(e^{2\delta|\gamma|}-1\bigr).
\end{equation}
In particular, if $z'=0$, then
\[
\left|
\log\Xi(z)-\sum_{\gamma\in\cP}z_\gamma
\right|
\le
\sum_{\gamma\in\cP}|z_\gamma|
\bigl(e^{2\delta|\gamma|}-1\bigr).
\]
\end{lemma}

\begin{proof}
Set $a(\gamma)=2\delta|\gamma|$. Since $2\delta\le1$,
\[
\sum_{\eta:\,\eta\cap\gamma\ne\varnothing}\bar z_\eta e^{a(\eta)}
\le
\sum_{x\in\gamma}\sum_{\eta\ni x}\bar z_\eta e^{|\eta|}
\le \delta|\gamma|\le a(\gamma).
\]
The first inequality may count an incompatible polymer once for each point of its intersection with $\gamma$, which is harmless for an upper bound. Thus Lemma~\ref{lem:KP} applies.

Differentiating \eqref{eq:cluster-log} with respect to $w_\gamma$ gives
\[
\frac{\partial}{\partial w_\gamma}\log\Xi(w)
=
\sum_{k\ge1}\frac1{(k-1)!}
\sum_{\gamma_2,\ldots,\gamma_k\in\cP}
\phi^{\mathrm T}(\gamma,\gamma_2,\ldots,\gamma_k)
\prod_{j=2}^k w_{\gamma_j}.
\]
The $k=1$ term is $1$. By \eqref{eq:rooted-cluster}, uniformly on $D(\bar z)$,
\begin{equation}\label{eq:derivative-bound}
\left|
\frac{\partial}{\partial w_\gamma}\log\Xi(w)-1
\right|
\le e^{2\delta|\gamma|}-1.
\end{equation}
Define $w(t)=(1-t)z'+tz$ for $0\le t\le1$ and
\[
F(t)=\log\Xi(w(t))-\sum_{\gamma\in\cP}w_\gamma(t).
\]
The segment lies in $D(\bar z)$. The chain rule and \eqref{eq:derivative-bound} give
\[
|F'(t)|
\le
\sum_{\gamma\in\cP}|z_\gamma-z'_\gamma|
\bigl(e^{2\delta|\gamma|}-1\bigr).
\]
Integrating from $0$ to $1$ proves \eqref{eq:polymer-perturbation}. Taking $z'=0$ and using $\Xi(0)=1$ proves the final assertion.
\end{proof}

\begin{lemma}\label{lem:differential-polymer}
Let $I$ be an interval and let $z(x)=(z_\gamma(x))_{\gamma\in\cP}$ be a continuously differentiable family of activities. At a point $x\in I$, define
\[
\delta(x):=\max_{v\in X}\sum_{\gamma\ni v}|z_\gamma(x)|e^{|\gamma|}.
\]
If $\delta(x)\le1/4$, then
\[
\left|
\frac{\dd}{\dd x}\log\Xi(z(x))-\sum_{\gamma\in\cP}z'_\gamma(x)
\right|
\le
\sum_{\gamma\in\cP}|z'_\gamma(x)|
\bigl(e^{2\delta(x)|\gamma|}-1\bigr).
\]
\end{lemma}

\begin{proof}
Since the polymer set is finite, the ordinary chain rule and \eqref{eq:derivative-bound} give
\begin{align*}
\left|
\frac{\dd}{\dd x}\log\Xi(z(x))-\sum_{\gamma\in\cP}z'_\gamma(x)
\right|
&=
\left|
\sum_{\gamma\in\cP}z'_\gamma(x)
\left(
\frac{\partial}{\partial z_\gamma}\log\Xi(z(x))-1
\right)
\right|\\
&\le
\sum_{\gamma\in\cP}|z'_\gamma(x)|
\bigl(e^{2\delta(x)|\gamma|}-1\bigr).
\end{align*}
\end{proof}

The following is Proposition~2.4 of Scott and Sokal \cite{ScottSokal}. See also Fern\'andez and Procacci \cite{FernandezProcacci}.

\begin{lemma}[Scott--Sokal tree bound \cite{ScottSokal}]\label{lem:tree-bound}
If $J$ is a graph and $a_e\in[0,1]$ for $e\in E(J)$, then
\[
\left|
\sum_{\substack{A\subseteq E(J)\\(V(J),~A)\text{ connected}}}
(-1)^{|A|}\prod_{e\in A}a_e
\right|
\le
\sum_{\substack{T\subseteq E(J)\\(V(J),~T)\text{ a tree}}}
\prod_{e\in T}a_e.
\]
Here $(V(J),A)$ is the spanning graph with the original vertex set and edge set $A$. Scott and Sokal state the result for complex edge weights $z_e$ satisfying $|1+z_e|\le1$. Taking $z_e=-a_e$ gives the displayed form.
\end{lemma}

\subsection{Graphon switching estimates}

A graphon is a symmetric measurable function $W:[0,1]^2\to[0,1]$. Define $U=1-W$ and
\begin{equation}\label{eq:q-rho}
q=\int_{[0,1]^2}U(x,y)\dd x\dd y,
\qquad
\rho=\sup_{x\in[0,1]}\int_0^1U(x,y)\dd y.
\end{equation}
Changing a graphon on a null set does not alter any density. We therefore choose an equivalent representative for which every row integral is defined and bounded by the essential supremum of the original graphon, so the ordinary supremum in \eqref{eq:q-rho} is valid. Since $q$ is the average row integral, $q\le\rho$.

For a graph $F$, write
\[
t_F(W)=
\int_{[0,1]^{V(F)}}
\prod_{ab\in E(F)}W(x_a,x_b)
\prod_{a\in V(F)}\dd x_a.
\]
Fix a graph $H$, and set $h=v(H)$ and $m=e(H)$. We apply the preceding polymer model with
\[
X=V(H),
\qquad
\cP=\{S\subseteq V(H): |S|\ge2\}.
\]
For $S\subseteq V(H)$ with $|S|\ge2$, define
\[
t_{(S,A)}(U)=
\int_{[0,1]^S}
\prod_{ab\in A}U(x_a,x_b)
\prod_{a\in S}\dd x_a,
\quad
z_{H,U}(S)=
\sum_{\substack{A\subseteq E(H[S])\\(S,~A)\text{ connected}}}
(-1)^{|A|}t_{(S,A)}(U).
\]

\begin{lemma}\label{lem:graphon-polymer}
For every graph $H$,
\begin{equation}\label{eq:graphon-polymer}
t_H(W)=
\sum_{\Gamma\text{ pairwise compatible}}
\left(\prod_{S\in\Gamma}z_{H,U}(S)\right).
\end{equation}
\end{lemma}

\begin{proof}
Expanding every factor $W=1-U$ gives
\begin{equation}\label{eq:graphon-IE}
t_H(W)=
\sum_{A\subseteq E(H)}(-1)^{|A|}
\int_{[0,1]^{V(H)}}
\prod_{ab\in A}U(x_a,x_b)
\prod_{a\in V(H)}\dd x_a.
\end{equation}
Fix $A\subseteq E(H)$. Let $S_1,\ldots,S_\ell$ be the vertex sets of the connected components of $(V(H),A)$ that contain an edge, and put $A_j=A\cap\binom{S_j}{2}$. Then
\[
A=A_1\dotcup\cdots\dotcup A_\ell,
\qquad
(-1)^{|A|}=\prod_{j=1}^\ell(-1)^{|A_j|},
\]
and Fubini's theorem gives
\[
\int_{[0,1]^{V(H)}}
\prod_{ab\in A}U(x_a,x_b)
\prod_a\dd x_a
=
\prod_{j=1}^\ell t_{(S_j,A_j)}(U).
\]
Conversely, pairwise disjoint sets $S_1,\ldots,S_\ell$ together with connected edge sets $A_j\subseteq E(H[S_j])$ determine the unique edge set $A=\dot\bigcup_jA_j$ whose nontrivial connected components are $(S_j,A_j)$. Substituting this bijection into \eqref{eq:graphon-IE} and summing first over every connected $A_j$ gives \eqref{eq:graphon-polymer}.
\end{proof}

The following estimates convert connected activities into geometric series.

\begin{lemma}\label{lem:tree-integrals}
Let $T$ be a tree on a vertex set $S$ of size $s\ge2$. Then
\begin{equation}\label{eq:tree-unrooted}
t_T(U)\le q\rho^{s-2}.
\end{equation}
For a fixed vertex $r\in S$ and $x\in[0,1]$,
\begin{equation}\label{eq:tree-one-root}
\int_{[0,1]^{S\setminus\{r\}}}
\prod_{ab\in E(T)}U(x_a,x_b)
\prod_{a\ne r}\dd x_a
\le\rho^{s-1},
\qquad x_r=x.
\end{equation}
For two fixed vertices $r_1,r_2\in S$ at tree-distance at least two and $x,y\in[0,1]$,
\begin{equation}\label{eq:tree-two-roots}
\int_{[0,1]^{S\setminus\{r_1,r_2\}}}
\prod_{ab\in E(T)}U(x_a,x_b)
\prod_{a\notin\{r_1,r_2\}}\dd x_a
\le\rho^{s-2},
\qquad x_{r_1}=x,\ x_{r_2}=y.
\end{equation}
\end{lemma}

\begin{proof}
For \eqref{eq:tree-unrooted}, retain one edge of $T$ and successively integrate all other vertices from the leaves inward. Each such integration contributes at most $\rho$, and the final retained edge contributes
\[
\int_{[0,1]^2}U(x,y)\dd x\dd y=q.
\]
This gives $q\rho^{s-2}$.

For \eqref{eq:tree-one-root}, root $T$ at $r$ and integrate the other vertices from the leaves toward the root. Each of the $s-1$ integrations contributes at most $\rho$.

For \eqref{eq:tree-two-roots}, first integrate every branch off the unique $r_1r_2$-path. Suppose that path has length $\ell\ge2$. The remaining path integral is the kernel convolution $U^{(\ell)}(x,y)$. Since $0\le U\le1$ and $U$ is symmetric,
\[
U^{(2)}(x,y)=\int_0^1U(x,z)U(z,y)\dd z
\le\int_0^1U(x,z)\dd z\le\rho.
\]
If $\ell\ge3$, induction gives
\[
U^{(\ell)}(x,y)
=\int_0^1U^{(\ell-1)}(x,z)U(z,y)\dd z
\le\rho^{\ell-2}\int_0^1U(z,y)\dd z
\le\rho^{\ell-1}.
\]
The branches attached to this path contribute the remaining factors of $\rho$, for a total of $\rho^{s-2}$.
\end{proof}

\begin{lemma}\label{lem:tree-counts}
  Let $H$ be a graph with $h$ vertices, $m$ edges, and maximum degree $\Delta$. For $s\ge 2$,
  consider pairs $(S,T)$ where $S\subseteq V(H)$ is an $s$-vertex subset and $T$ is a spanning tree of $H[S]$.
\begin{enumerate}
\item If a fixed vertex must belong to $S$, the number of pairs $(S,T)$ is at most $(4\Delta)^{s-1}$.
\item If $s\ge3$, the total number of pairs $(S,T)$ is at most $8m(4\Delta)^{s-2}$.
\item If two fixed vertices must belong to $S$, the number of pairs is at most
\[
\binom{h-2}{s-2}s^{s-2}\le(3eh)^{s-2}.
\]
\end{enumerate}
\end{lemma}

\begin{proof}
For the first assertion, root $T$ at the prescribed vertex and order the children of every vertex. The number of rooted plane-tree shapes with $s$ vertices is the Catalan number $C_{s-1}<4^{s-1}$. Embed a fixed shape in an order in which each parent precedes its children. Once the root is fixed, every nonroot vertex has at most $\Delta$ possible images, since it must be mapped to a neighbor of its parent's image. Ignoring injectivity and multiple encodings only enlarges the count, so the total is at most $(4\Delta)^{s-1}$.

For the second assertion, distinguish and orient one edge of $T$. There are at most $2m$ choices for the image of this oriented edge in $H$. Deleting it leaves an ordered pair of rooted tree components. After ordering the children in both components, the Catalan convolution gives
\[
\sum_{a=1}^{s-1}C_{a-1}C_{s-a-1}=C_{s-1}\le4^{s-1}
\]
possible pairs of plane-tree shapes. Once the two roots are fixed by the oriented edge, each of the remaining $s-2$ vertices has at most $\Delta$ choices when its parent has already been embedded. Thus the number of distinguished encodings is at most
\[
2m\,4^{s-1}\Delta^{s-2}=8m(4\Delta)^{s-2}.
\]
Every desired pair has at least one such encoding, which proves the assertion.

For the third assertion, choose the other $s-2$ vertices of $S$ in $\binom{h-2}{s-2}$ ways. The number of trees on the labeled set $S$ is $s^{s-2}$ by Cayley's formula, so this also bounds the number of spanning trees of $H[S]$. Finally,
\[
\binom{h-2}{s-2}s^{s-2}
\le
\left(\frac{e(h-2)}{s-2}\right)^{s-2}s^{s-2}
\le(3eh)^{s-2},
\]
because $s/(s-2)\le3$ for $s\ge3$. The case $s=2$ is immediate.
\end{proof}

\begin{lemma}\label{lem:activity-bounds}
There are absolute constants $c_a,C_a>0$ such that, if $h\rho\le c_a$, then
\begin{align}
\max_{v\in V(H)}\sum_{S\ni v}|z_{H,U}(S)|e^{|S|}
&\le C_ah\rho,\label{eq:activity-local}\\
\sum_{|S|\ge3}|z_{H,U}(S)|e^{|S|}
&\le C_amq h\rho.\label{eq:activity-high}
\end{align}
Moreover, for each edge $uv\in E(H)$,
\begin{equation}\label{eq:activity-edge}
\sum_{\substack{S\supseteq\{u,v\}\\|S|\ge3}}
\bigl(|z_{H,U}(S)|+|z_{H-uv,U}(S)|\bigr)e^{|S|}
\le C_a q h\rho.
\end{equation}
\end{lemma}

\begin{proof}
For $|S|=s$, Lemma~\ref{lem:tree-bound} applied pointwise with $a_{ab}=U(x_a,x_b)$, followed by Lemma~\ref{lem:tree-integrals}, gives the complete chain
\begin{align*}
|z_{H,U}(S)|
&\le
\int_{[0,1]^S}
\left|
\sum_{\substack{A\subseteq E(H[S])\\(S,A)\text{ connected}}}
(-1)^{|A|}\prod_{ab\in A}U(x_a,x_b)
\right|
\prod_{a\in S}\dd x_a\\
&\le
\sum_{T\in\cT(H[S])}
\int_{[0,1]^S}
\prod_{ab\in E(T)}U(x_a,x_b)
\prod_{a\in S}\dd x_a\\
&\le |\cT(H[S])|q\rho^{s-2}.
\end{align*}
For a fixed vertex, the contribution of all sets of size $s$ to \eqref{eq:activity-local} is at most
\[
(4\Delta)^{s-1}q\rho^{s-2}e^s
=4e^2\Delta q(4e\Delta\rho)^{s-2}.
\]
Summing this expression over $s\ge2$, using $q\le\rho$ and $\Delta\le h$, proves \eqref{eq:activity-local}.

For \eqref{eq:activity-high}, the contribution of all sets of size $s\ge3$ is at most
\[
8mq(4\Delta)^{s-2}\rho^{s-2}e^s
=8e^2mq(4e\Delta\rho)^{s-2}.
\]
Summing over $s\ge3$ gives $O(mq\Delta\rho)=O(mqh\rho)$.

For \eqref{eq:activity-edge}, Lemma~\ref{lem:tree-counts}(3) gives, for each $s\ge3$ and for each of $H$ and $H-uv$,
\[
\sum_{\substack{S\supseteq\{u,v\}\\|S|=s}}|z_{H,U}(S)|e^s
\le e^2q(3e^2h\rho)^{s-2}.
\]
Summing this over $s\ge3$ for both graphs proves \eqref{eq:activity-edge} after choosing $c_a$ so that all three geometric ratios are at most $1/2$.
\end{proof} 

The preceding activity bounds give the two edge-switching estimates used below. We begin with the effect of deleting an edge of $H$.

\begin{theorem}\label{thm:source-edge-switch}
There are absolute constants $c,C>0$ such that, if $h\rho\le c$, then every $uv\in E(H)$ satisfies
\[
\left|
\log\frac{t_{H-uv}(W)}{t_H(W)}-q
\right|
\le Cqh\rho,
\qquad
\frac{t_{H-uv}(W)-t_H(W)}{t_H(W)}
=q\bigl(1+O(h\rho)\bigr).
\]
\end{theorem}

\begin{proof}
Set $z_S=z_{H,U}(S)$ and $z'_S=z_{H-uv,U}(S)$. Then $z_S=z'_S$ unless $S\supseteq\{u,v\}$. For $S=\{u,v\}$, $z_S=-q$ and $z'_S=0$. Hence
\[
\sum_S(z'_S-z_S)
=q+
\sum_{\substack{S\supseteq\{u,v\}\\|S|\ge3}}(z'_S-z_S),
\]
and \eqref{eq:activity-edge} gives
\begin{equation}\label{eq:edge-activity-difference}
\sum_{\substack{S\supseteq\{u,v\}\\|S|\ge3}}|z'_S-z_S|e^{|S|}
\le Cqh\rho.
\end{equation}
By \eqref{eq:activity-local}, the local norm in Lemma~\ref{lem:polymer-perturbation} satisfies $\delta\le Ch\rho\le1/4$. Therefore
\begin{align*}
\left|
\log\frac{t_{H-uv}(W)}{t_H(W)}-q
\right|
&\le
\sum_{\substack{S\supseteq\{u,v\}\\|S|\ge3}}|z'_S-z_S|
+q(e^{4\delta}-1)\\
&\quad+
\sum_{\substack{S\supseteq\{u,v\}\\|S|\ge3}}|z'_S-z_S|
\bigl(e^{2\delta|S|}-1\bigr)\\
&\le Cqh\rho.
\end{align*}
Here the last line uses \eqref{eq:edge-activity-difference}, $e^{2\delta|S|}-1\le e^{|S|}$, and $e^{4\delta}-1=O(h\rho)$. This proves the logarithmic estimate. Since $q\le\rho$ and $h\ge2$,
\[
e^{q+O(qh\rho)}-1
=q+O(qh\rho)+O(q^2)
=q\bigl(1+O(h\rho)\bigr),
\]
which proves the second assertion.
\end{proof}

We also record the corresponding estimate when the images of the endpoints of an edge are prescribed.

Fix an oriented edge $(u,v)$ of $H$. For $x,y\in[0,1]$, define
\begin{equation}\label{eq:Psi}
\Psi^W_{uv}(x,y)=
\int_{[0,1]^{V(H)\setminus\{u,v\}}}
\prod_{ab\in E(H)\setminus\{uv\}}W(x_a,x_b)
\prod_{a\notin\{u,v\}}\dd x_a,
\qquad x_u=x,\ x_v=y.
\end{equation}
Thus $\Psi^W_{uv}(x,y)$ is the homomorphism density of $H-uv$ with the images of $u$ and $v$ fixed at $x$ and $y$.

\begin{theorem}\label{thm:fixed-endpoint}
There are absolute constants $c,C>0$ such that, if $h\rho\le c$, then for every $x,y\in[0,1]$,
\[
e^{-Ch\rho}t_H(W)
\le\Psi^W_{uv}(x,y)
\le e^{Ch\rho}t_H(W).
\]
\end{theorem}

\begin{proof}
For $S\subseteq V(H)$ with $|S|\ge2$, define the fixed-endpoint activity
\[
\widehat z_{x,y}(S)=
\sum_{\substack{A\subseteq E((H-uv)[S])\\(S,A)\text{ connected}}}
(-1)^{|A|}
\int_{[0,1]^{S\setminus\{u,v\}}}
\prod_{ab\in A}U(x_a,x_b)
\prod_{a\in S\setminus\{u,v\}}\dd x_a,
\]
where $x_u=x$ and $x_v=y$ whenever those labels lie in $S$. Expanding $\Psi^W_{uv}(x,y)$ in $U=1-W$ and grouping the connected components exactly as in the proof of Lemma~\ref{lem:graphon-polymer} gives
\[
\Psi^W_{uv}(x,y)=\Xi\bigl((\widehat z_{x,y}(S))_S\bigr).
\]
A component meeting neither fixed vertex has activity $z_{H,U}(S)$. If a component meets exactly one of $u,v$, Lemmas~\ref{lem:tree-bound}, \ref{lem:tree-integrals}, and \ref{lem:tree-counts}(1) give
\[
\sum_{|S\cap\{u,v\}|=1}|\widehat z_{x,y}(S)|e^{|S|}
\le2e\sum_{j\ge1}(4e\Delta\rho)^j
=O(\Delta\rho)=O(h\rho).
\]
If a component contains both fixed vertices, the direct edge $uv$ is absent. Their distance in every spanning tree is therefore at least two, and Lemmas~\ref{lem:tree-integrals} and \ref{lem:tree-counts}(3) give
\[
\sum_{S\supseteq\{u,v\}}|\widehat z_{x,y}(S)|e^{|S|}
\le e^2\sum_{j\ge1}(3e^2h\rho)^j
=O(h\rho).
\]
The unrestricted and fixed-endpoint systems are equal on polymers disjoint from $\{u,v\}$. Combining the two displays with \eqref{eq:activity-local},
\begin{align*}
\sum_{S:\,S\cap\{u,v\}\ne\varnothing}
|z_{H,U}(S)-\widehat z_{x,y}(S)|e^{|S|}
&\le
\sum_{S:\,S\cap\{u,v\}\ne\varnothing}
\bigl(|z_{H,U}(S)|+|\widehat z_{x,y}(S)|\bigr)e^{|S|}\\
&\le Ch\rho.
\end{align*}
The same estimates give a local norm at most $Ch\rho\le1/4$ for the coordinatewise maximum of the two activity systems. Lemma~\ref{lem:polymer-perturbation} now yields
\[
\left|\log\Psi^W_{uv}(x,y)-\log t_H(W)\right|\le Ch\rho.
\]
Both quantities are nonnegative integrals and are nonzero by Lemma~\ref{lem:KP}. Hence they are positive, and exponentiation proves the theorem.
\end{proof}

We finish the section by transferring these graphon estimates to finite graphs.

\subsection{Transfer to finite graphs}

For an $n$-vertex graph $G$, split $[0,1]$ into equal intervals $I_x$ indexed by $x\in V(G)$. Its step graphon is
\[
W_G(s,t)=\1_{xy\in E(G)}
\qquad\text{when }(s,t)\in I_x\times I_y.
\]
Consequently,
\begin{equation}\label{eq:step-graphon}
t_F(W_G)=\frac{\homc(F,G)}{n^{v(F)}}.
\end{equation}
The diagonal blocks are zero because $G$ has no loops. Hence
\[
\rho(G):=\rho(W_G)=\max_{x\in V(G)}\frac{n-d_G(x)}{n}.
\]
The term $n-d_G(x)$ includes $x$ itself. This is the normalized maximum missing degree, not the spectral radius of $G$.

The meaning of \eqref{eq:Psi} for a step graphon is exact:
\begin{equation}\label{eq:Psi-finite}
n^{h-2}\Psi^{W_G}_{uv}(x,y)
=\homc_{u\mapsto x,\,v\mapsto y}(H-uv,G).
\end{equation}
By \eqref{eq:collision}, the corresponding injective count with fixed images of $u$ and $v$ differs from the right-hand side by $O_H(n^{h-3})$.

\begin{corollary}\label{cor:finite-switching}
There are absolute constants $c,C>0$ such that, if $h\rho(G)\le c$, then uniformly in the chosen $x,y\in V(G)$:
\begin{enumerate}
\item if $xy\notin E(G)$, then
\[
\log\frac{\inj(H,G+xy)}{\inj(H,G)}
=\frac{2m}{n^2}
\bigl(1+O(h\rho(G))+O_H(n^{-1})\bigr);
\]
\item if $xy\in E(G)$, then
\[
\log\frac{\inj(H,G)}{\inj(H,G-xy)}
=\frac{2m}{n^2}
\bigl(1+O(h(\rho(G)+n^{-1}))+O_H(n^{-1})\bigr).
\]
\end{enumerate}
\end{corollary}

\begin{proof}
Assume $m>0$. Equations \eqref{eq:step-graphon} and \eqref{eq:collision} give
\begin{equation}\label{eq:inj-step}
\inj(H,G)=n^h t_H(W_G)+O_H(n^{h-1}).
\end{equation}
We record a lower bound for its leading factor. The two-vertex activities have total $-mq$. By \eqref{eq:activity-high},
\[
\left|\sum_{|S|\ge3}z_{H,U}(S)\right|\le Cmqh\rho(G).
\]
The nonlinear error in Lemma~\ref{lem:polymer-perturbation} is at most
\[
mq(e^{4\delta}-1)
+\sum_{|S|\ge3}|z_{H,U}(S)|\bigl(e^{2\delta|S|}-1\bigr)
\le Cmqh\rho(G).
\]
Thus
\[
|\log t_H(W_G)+mq|\le Cmqh\rho(G).
\]
In particular, $t_H(W_G)$ is bounded below by a positive constant depending only on the fixed $H$, so the error in \eqref{eq:inj-step} is relatively $O_H(n^{-1})$.

We also verify positivity before taking logarithms. Since $\rho(G)\ge1/n$, when greedily embedding a new vertex of $H$, the already used vertices and the nonneighbors of the images of its already embedded neighbors exclude at most
\[
(h-1)+(h-1)n\rho(G)\le2nh\rho(G)<n
\]
host vertices after decreasing $c$. Hence $\inj(H,G)>0$. The same argument applies after deleting one edge.

Suppose first that $xy\notin E(G)$. Every new injective map uses $xy$, and injectivity selects a unique oriented edge $(u,v)$ of $H$ mapped to $(x,y)$. By \eqref{eq:Psi-finite},
\[
\inj_{u\mapsto x,\,v\mapsto y}(H-uv,G)
=n^{h-2}\Psi^{W_G}_{uv}(x,y)+O_H(n^{h-3}).
\]
Summing over the $2m$ oriented edges and applying Theorem~\ref{thm:fixed-endpoint},
\[
\inj(H,G+xy)-\inj(H,G)
=2mn^{h-2}t_H(W_G)
\bigl(1+O(h\rho(G))+O_H(n^{-1})\bigr).
\]
Divide by \eqref{eq:inj-step} and use $\log(1+s)=s+O(s^2)$, where $s=O_H(n^{-2})$, to prove the first assertion.

For the second assertion, apply the first one to $G-xy$. Deleting one edge increases each affected missing degree by one, so $\rho(G-xy)\le\rho(G)+1/n$. Decreasing $c$ once makes the first assertion applicable and gives the claimed formula.
\end{proof}

\section{Color profiles and multipartite balancing}\label{sec:profiles}

Throughout this section, let $H$ be a graph with $h=v(H)$ and $m=e(H)$. A probability vector is a vector $p=(p_1,\ldots,p_r)$ with $p_i\ge0$ and $\sum_i p_i=1$.

In an independent $p$-coloring, each $v\in V(H)$ receives a color $c(v)\in[r]$ independently, with $\Prb(c(v)=i)=p_i$. It is proper if adjacent vertices receive different colors. Define
\[
\Phi_H(p)=
\sum_{\substack{c:V(H)\to[r]\\c\text{ proper}}}
\left(\prod_{v\in V(H)}p_{c(v)}\right).
\]
Thus $\Phi_H(p)$ is exactly the probability that an independent $p$-coloring is proper. Set
\[
s_k(p)=\sum_i p_i^k.
\]
For $S\subseteq V(H)$ with $|S|\ge2$, define
\[
\chi_H(S)=
\sum_{\substack{A\subseteq E(H[S])\\(S,A)\text{ connected}}}(-1)^{|A|},
\qquad
z_p(S)=\chi_H(S)s_{|S|}(p).
\]
The sets $S\subseteq V(H)$ of size at least two are polymers, compatible exactly when they are disjoint.

\begin{lemma}\label{lem:profile-polymer}
For every probability vector $p$,
\begin{equation}\label{eq:profile-polymer}
\Phi_H(p)=
\sum_{\Gamma\text{ pairwise disjoint}}
\prod_{S\in\Gamma}z_p(S).
\end{equation}
Moreover,
\[
|\chi_H(S)|\le|\cT(H[S])|.
\]
\end{lemma}

\begin{proof}
For $e=uv\in E(H)$, let $B_e$ be the event $c(u)=c(v)$. Inclusion--exclusion gives
\[
\Phi_H(p)=
\sum_{A\subseteq E(H)}(-1)^{|A|}
\Prb\!\left(\bigcap_{e\in A}B_e\right).
\]
Fix $A$ and let $S_1,\ldots,S_\ell$ be the vertex sets of the connected components of $(V(H),A)$ containing an edge. Put $A_j=A\cap\binom{S_j}{2}$. The equalities imposed by the edges in $A$ force all vertices of $S_j$ to have one common color, so
\[
\Prb\!\left(\bigcap_{e\in A}B_e\right)
=\prod_{j=1}^\ell\sum_{i=1}^r p_i^{|S_j|}
=\prod_{j=1}^\ell s_{|S_j|}(p),
\]
and
\[
A=A_1\dotcup\cdots\dotcup A_\ell,
\qquad
(-1)^{|A|}=\prod_{j=1}^\ell(-1)^{|A_j|}.
\]
Conversely, pairwise disjoint sets $S_1,\ldots,S_\ell$ and connected choices $A_j\subseteq E(H[S_j])$ satisfy
\[
A=\dot\bigcup_{j=1}^\ell A_j,
\text{and}~
\{\text{nontrivial components of }(V(H),A)\}
=\{(S_j,A_j):1\le j\le\ell\}.
\]
Therefore
\begin{align*}
\Phi_H(p)
&=\sum_{\Gamma\text{ pairwise disjoint}}
\prod_{S\in\Gamma}
\left(
\sum_{\substack{A\subseteq E(H[S])\\(S,A)\text{ connected}}}
(-1)^{|A|}s_{|S|}(p)
\right)\\
&=\sum_{\Gamma\text{ pairwise disjoint}}
\prod_{S\in\Gamma}z_p(S),
\end{align*}
which yields \eqref{eq:profile-polymer}. The last assertion follows from Lemma~\ref{lem:tree-bound} with all edge weights equal to $1$.
\end{proof}

\subsection{Continuous profile balancing}

Suppose $p_i\ge p_j$, and move mass $t$ from coordinate $i$ to coordinate $j$:
\[
p'_i=p_i-t,
\qquad
p'_j=p_j+t,
\qquad
0<t\le\frac{p_i-p_j}{2}.
\]
All other coordinates are unchanged. Define
\[
\rho=\max_k p_k,
\qquad
\Delta_2=s_2(p)-s_2(p')=2t(p_i-p_j-t)>0.
\]
The transfer moves the two coordinates toward equality, and $\max_k p'_k\le\rho$.

\begin{lemma}\label{lem:power-sum}
For every integer $k\ge2$,
\[
0\le s_k(p)-s_k(p')
\le\binom{k}{2}\rho^{k-2}\Delta_2.
\]
\end{lemma}

\begin{proof}
Write $p_i=c+x$ and $p_j=c-x$, where $x\ge t$, and define
\[
g_k(u)=(c+u)^k+(c-u)^k.
\]
Then
\[
s_k(p)-s_k(p')=\int_{x-t}^x g'_k(u)\dd u.
\]
For $0\le u\le x$,
\[
0\le g'_k(u)
=k\bigl((c+u)^{k-1}-(c-u)^{k-1}\bigr)
\le2k(k-1)u\rho^{k-2}.
\]
Consequently,
\[
s_k(p)-s_k(p')
\le\int_{x-t}^x2k(k-1)u\rho^{k-2}\dd u
=\binom{k}{2}\rho^{k-2}\Delta_2,
\]
because $\Delta_2=2(x^2-(x-t)^2)$.
\end{proof}

\begin{theorem}\label{thm:profile-balancing}
There are absolute constants $c,C>0$ such that, if $h\rho\le c$, then
\begin{equation}\label{eq:profile-balancing}
\left|
\log\frac{\Phi_H(p')}{\Phi_H(p)}-m\Delta_2
\right|
\le Cmh\rho\Delta_2.
\end{equation}
In particular,
\[
\log\frac{\Phi_H(p')}{\Phi_H(p)}
\ge m\Delta_2(1-Ch\rho).
\]
\end{theorem}

\begin{proof}
For $q\in\{p,p'\}$, every coordinate is at most $\rho$, so $s_s(q)\le\rho^{s-1}$. Lemmas~\ref{lem:tree-bound} and \ref{lem:profile-polymer} give
\[
|z_q(S)|\le|\cT(H[S])|\rho^{|S|-1}.
\]
By (1) of Lemma~\ref{lem:tree-counts},
\[
\max_{v\in V(H)}
\sum_{S\ni v}\max\{|z_p(S)|,|z_{p'}(S)|\}e^{|S|}
\le e\sum_{j\ge1}(4e\Delta\rho)^j
\le C'\Delta\rho\le C'h\rho
\]
for an absolute constant $C'$. Choose $c$ so that $C'h\rho\le1/4$. Thus Lemma~\ref{lem:polymer-perturbation} applies with $\delta\le C'h\rho$.

A two-vertex activity changes by $\Delta_2$ precisely when its two vertices form an edge of $H$. Their total first-order change is $m\Delta_2$. For $s=|S|\ge3$, Lemma~\ref{lem:power-sum} gives
\[
|z_{p'}(S)-z_p(S)|
\le|\cT(H[S])|\binom{s}{2}\rho^{s-2}\Delta_2.
\]
Using (2) of Lemma~\ref{lem:tree-counts},
\begin{equation}\label{eq:profile-high-difference}
\sum_{|S|\ge3}|z_{p'}(S)-z_p(S)|e^{|S|}
\le Cm\Delta_2\sum_{j\ge1}\binom{j+2}{2}(4e\Delta\rho)^j
\le Cm\Delta\rho\Delta_2
\le Cmh\rho\Delta_2.
\end{equation}
The three error terms in Lemma~\ref{lem:polymer-perturbation} satisfy
\begin{align*}
\left|
\log\frac{\Phi_H(p')}{\Phi_H(p)}-m\Delta_2
\right|
&\le
\left|\sum_{|S|\ge3}(z_{p'}(S)-z_p(S))\right|
+m\Delta_2(e^{4\delta}-1)\\
&\quad+
\sum_{|S|\ge3}|z_{p'}(S)-z_p(S)|
\bigl(e^{2\delta|S|}-1\bigr)\\
&\le Cmh\rho\Delta_2.
\end{align*}
Here the first and third terms are bounded by \eqref{eq:profile-high-difference}, using $e^{2\delta|S|}-1\le e^{|S|}$, and the middle term is $O(mh\rho\Delta_2)$. This proves \eqref{eq:profile-balancing}. Nonvanishing follows from Lemma~\ref{lem:KP}, and both partition functions are positive probabilities.
\end{proof}

\subsection{Finite balancing and coloring consequences}

Let $a=(a_1,\ldots,a_r)\in\mathbb Z_{\ge0}^r$ with $\sum_i a_i=n$, set $p_i=a_i/n$, and let $K_a=K_{a_1,\ldots,a_r}$. For a partition $\pi$ of $V(H)$, the quotient $H/\pi$ has the blocks of $\pi$ as vertices. An edge inside a block becomes a loop. Since a loop cannot map to a simple graph, we set $\homc(H/\pi,G)=0$ in that case. Parallel quotient edges may be suppressed. Define
\[
\mu(\pi)=\prod_{B\in\pi}(-1)^{|B|-1}(|B|-1)!.
\]

\begin{lemma}\label{lem:mobius-inversion}
For every pair of graphs $H,G$,
\begin{equation}\label{eq:mobius-inversion}
\inj(H,G)=\sum_\pi\mu(\pi)\homc(H/\pi,G),
\end{equation}
where the sum is over all partitions of $V(H)$.
\end{lemma}

\begin{proof}
If $\sigma$ is a partition of $V(H)$, every homomorphism $H/\sigma\to G$ has a unique kernel partition $\pi$ such that for every $A\in \sigma$ implies $A\in \pi$. Contracting its equal fibers leaves an injective map $H/\pi\to G$. Conversely, each such injective map gives one homomorphism with kernel $\pi$. Hence
\[
\homc(H/\sigma,G)=\sum_{\pi\ge\sigma}\inj(H/\pi,G).
\]
The partition-lattice M\"obius inversion formula \cite{Stanley}, applied with $\sigma$ the singleton partition, gives \eqref{eq:mobius-inversion} and the displayed formula for $\mu(\pi)$.
\end{proof}

For loopless $J$,
\[
\homc(J,K_a)
=\sum_{\substack{c:V(J)\to[r]\\c\text{ proper}}}
\prod_{x\in V(J)}a_{c(x)}
=n^{v(J)}\Phi_J(p).
\]
For $1\le k<h$, define
\[
F_{H,k}(p)=
\sum_{\pi:\,|\pi|=h-k}\mu(\pi)\Phi_{H/\pi}(p),
\]
where looped terms are zero. Lemma~\ref{lem:mobius-inversion} gives the exact identity
\begin{equation}\label{eq:finite-mobius}
n^{-h}\inj(H,K_a)
=\Phi_H(p)+\sum_{k=1}^{h-1}n^{-k}F_{H,k}(p).
\end{equation}
Every $F_{H,k}$ is a symmetric polynomial.

\begin{lemma}\label{lem:even-function}
Let $f$ be an even twice continuously differentiable function on $[-x_0,x_0]$. Suppose $|f''(x)|\le2M$ on this interval. If
\[
x_0=\frac{d+1}{2n},
\qquad
x_1=\frac{d-1}{2n}
\]
for an integer $d\ge1$, then
\[
|f(x_0)-f(x_1)|\le M\frac{d}{n^2}.
\]
\end{lemma}

\begin{proof}
Evenness gives $f'(0)=0$. Thus $|f'(x)|\le2M|x|$, and
\[
|f(x_0)-f(x_1)|
\le\int_{x_1}^{x_0}2Mx\dd x
=M(x_0^2-x_1^2)
=M\frac{d}{n^2}.
\]
\end{proof}

\begin{theorem}\label{thm:finite-balancing}
There are absolute constants $c,C>0$ such that the following holds. Suppose
\[
\max_i p_i\le\frac{10}{r},
\qquad
\frac{h}{r}\le c,
\]
and $K'$ is obtained from $K_a$ by moving one vertex from a part of size $a_i$ to a part of size $a_j\le a_i-2$. Set $d=a_i-a_j-1$, then
\begin{equation}\label{eq:finite-balancing}
\log\frac{\inj(H,K')}{\inj(H,K_a)}
\ge\frac{2md}{n^2}
\left(1-C\frac{h}{r}-\frac{C_{H,r}}{n}\right)
\end{equation}
for all sufficiently large $n$.
\end{theorem}

\begin{proof}
Assume first that $m\ge1$. Since $p_i-p_j=(d+1)/n$ and $d\ge1$, the transfer has size $1/n\le(p_i-p_j)/2$, and
\[
\Delta_2=\frac{2d}{n^2}.
\]
Theorem~\ref{thm:profile-balancing} gives
\begin{equation}\label{eq:finite-leading}
\log\frac{\Phi_H(p')}{\Phi_H(p)}
\ge\frac{2md}{n^2}\left(1-C\frac{h}{r}\right).
\end{equation}
Let $D_r$ be the set of probability vectors $q$ satisfying $\max_iq_i\le10/r$. Shrink $c$ so that $h/r\le c$ implies $r\ge20h$. Every $q\in D_r$ has at least $h$ coordinates at least $1/(2r)$. Otherwise its total mass would be less than
\[
(h-1)\frac{10}{r}+(r-h+1)\frac1{2r}<1.
\]
Assigning distinct such colors gives the uniform lower bound
\begin{equation}\label{eq:Phi-lower}
\Phi_H(q)\ge h!(2r)^{-h}.
\end{equation}
Define
\[
R_n(q)=
\frac{\sum_{k=1}^{h-1}n^{-k}F_{H,k}(q)}{\Phi_H(q)}.
\]
The functions $F_{H,k}$ and $\Phi_H$ are fixed polynomials, and \eqref{eq:Phi-lower} keeps the denominator away from zero. Hence, for every $q\in D_r$ and every two-coordinate line contained in $D_r$, the absolute value of the second derivative of $R_n$ along that line is at most $2C_{H,r}/n$. The term $k=1$ is $O_{H,r}(n^{-1})$, and the terms $k\ge2$ are smaller.

Keep all coordinates except $i,j$ fixed and put
\[
f(x)=R_n(q(x)),
\qquad
q_i(x)=\frac{p_i+p_j}{2}+x,
\qquad
q_j(x)=\frac{p_i+p_j}{2}-x.
\]
Symmetry of $R_n$ makes $f$ even. The old and new values of $x$ are $(d+1)/(2n)$ and $(d-1)/(2n)$, respectively. The entire segment lies in $D_r$ because $D_r$ is convex and balancing does not increase the largest coordinate. Lemma~\ref{lem:even-function} gives
\[
|R_n(p')-R_n(p)|\le C_{H,r}\frac{d}{n^3}.
\]
For large $n$, $|R_n|\le1/2$. Therefore
\begin{equation}\label{eq:Rn-log}
|\log(1+R_n(p'))-\log(1+R_n(p))|
\le2|R_n(p')-R_n(p)|
\le2C_{H,r}\frac{d}{n^3}.
\end{equation}
The exact factorization from \eqref{eq:finite-mobius} is
\[
n^{-h}\inj(H,K_a)=\Phi_H(p)(1+R_n(p)).
\]
Combining \eqref{eq:finite-leading} and \eqref{eq:Rn-log} yields \eqref{eq:finite-balancing}. If $m=0$, every $n$-vertex host has $\inj(H,G)=(n)_h$, and the assertion is immediate.
\end{proof}

\begin{corollary}\label{cor:discrete-balancing}
Let $B=e(T_r(n))-e(K_a)$. If $\max_i a_i\le10n/r$ and $h/r\le c$, then
\begin{equation}\label{eq:discrete-balancing}
\log\frac{\inj(H,T_r(n))}{\inj(H,K_a)}
\ge\frac{2mB}{n^2}
\left(1-C\frac{h}{r}-\frac{C_{H,r}}{n}\right).
\end{equation}
\end{corollary}

\begin{proof}
Starting from $a^{(0)}=a$, whenever $a_i^{(t)}\ge a_j^{(t)}+2$, move one vertex from part $i$ to part $j$ and let
\[
d_t=a_i^{(t)}-a_j^{(t)}-1.
\]
For the two parts involved, the larger part remains part $i$ after the move because
\[
a_j^{(t)}+1\le a_i^{(t)}-1.
\]
Its size decreases by one, while every other part is unchanged. Hence the largest part size never increases. Also,
\[
(a_i^{(t)}-1)^2+(a_j^{(t)}+1)^2-(a_i^{(t)})^2-(a_j^{(t)})^2
=-2d_t<0,
\]
so the integer potential $\sum_k(a_k^{(t)})^2$ decreases and the process terminates exactly at a balanced profile $a^{(N)}$.

At step $t$, the edge gain is
\[
e(K_{a^{(t+1)}})-e(K_{a^{(t)}})=d_t.
\]
Therefore
\[
\sum_{t=0}^{N-1}d_t=e(T_r(n))-e(K_a)=B.
\]
Applying Theorem~\ref{thm:finite-balancing} at every step and summing gives
\begin{align*}
\log\frac{\inj(H,T_r(n))}{\inj(H,K_a)}
&=\sum_{t=0}^{N-1}
\log\frac{\inj(H,K_{a^{(t+1)}})}{\inj(H,K_{a^{(t)}})}\\
&\ge\frac{2m}{n^2}
\left(1-C\frac{h}{r}-\frac{C_{H,r}}{n}\right)
\sum_{t=0}^{N-1}d_t,
\end{align*}
which yields \eqref{eq:discrete-balancing}.
\end{proof}

We conclude the section with a consequence for the proper-coloring probability. For a graph $J$, let
\[
p_r(J)=
\frac{|\{c:V(J)\to[r]:c\text{ proper}\}|}{r^{v(J)}}
=\Phi_J(1/r,\ldots,1/r).
\]

\begin{corollary}\label{cor:color-deletion}
There is an absolute $C$ such that, if $r\ge Cv(J)$ and $D\subseteq E(J)$, then
\begin{equation}\label{eq:color-deletion}
\log\frac{p_r(J-D)}{p_r(J)}
\le |D|\left(\frac1r+C\frac{v(J)}{r^2}\right).
\end{equation}
\end{corollary}

\begin{proof}
Partition $[0,1]$ into $r$ equal intervals and let $W_r$ be zero on each diagonal block and one on every off-diagonal block. If $x_v$ lies in block $c(v)$, then
\[
\prod_{uv\in E(J)}W_r(x_u,x_v)
=\1_{\{c\text{ is a proper }r\text{-coloring of }J\}}.
\]
Hence $t_J(W_r)=p_r(J)$. The parameters in \eqref{eq:q-rho} are $q=\rho=1/r$. Enumerate $D=\{e_1,\ldots,e_s\}$ and put $J_\ell=J-\{e_1,\ldots,e_\ell\}$. Theorem~\ref{thm:source-edge-switch} gives
\[
\log\frac{p_r(J_\ell)}{p_r(J_{\ell-1})}
\le\frac1r+C\frac{v(J)}{r^2}.
\]
Summing over all admissible $\ell$ yields \eqref{eq:color-deletion}.
\end{proof}

\section{Proof of Theorem~\ref{thm:chromatic}}\label{sec:chromatic-proof}

Let $H$ be a graph with
\[
h:=v(H),
\qquad
m:=e(H),
\qquad
\Delta:=\Delta(H),
\]
and let $P_H(x)$ be its chromatic polynomial. Recall from Lemma~\ref{lem:profile-polymer} that, for $S\subseteq V(H)$ with $|S|\ge2$,
\[
\chi_H(S)=
\sum_{\substack{A\subseteq E(H[S])\\(S,A)\text{ connected}}}
(-1)^{|A|},
\qquad
|\chi_H(S)|\le|\cT(H[S])|.
\]
For real $x\ne0$, define
\[
z_x(S):=\chi_H(S)x^{1-|S|}.
\]

\begin{lemma}\label{lem:chromatic-polymer}
For every real $x\ne0$,
\begin{equation}\label{eq:chromatic-polymer}
\frac{P_H(x)}{x^h}
=
\sum_{\substack{\Gamma:\,\Gamma\text{ is a family of pairwise}\\
                  \text{disjoint subsets of }V(H),~\\
                   \text{and }|S|\ge2\text{ for every }S\in\Gamma}}\left(
\prod_{S\in\Gamma}z_x(S)\right).
\end{equation}
The identity is an identity of Laurent polynomials in $x$.
\end{lemma}

\begin{proof}
The Fortuin--Kasteleyn formula \cite{FortuinKasteleyn} is
\[
P_H(x)=\sum_{A\subseteq E(H)}(-1)^{|A|}x^{k(A)},
\]
where $k(A)$ is the number of connected components of the spanning graph $(V(H),A)$, including isolated vertices. Let $S_1,\ldots,S_t$ be the vertex sets of its nontrivial connected components. Since the remaining $h-\sum_i|S_i|$ vertices are isolated,
\[
x^{k(A)-h}=\prod_{i=1}^t x^{1-|S_i|}.
\]
The sets $S_1,\ldots,S_t$ are pairwise disjoint. Conversely, after these sets are fixed, the choice of $A$ factors into an independent choice of a connected spanning edge set of $H[S_i]$ for every $i$. Summing $(-1)^{|A|}$ over each component produces $\chi_H(S_i)$. Dividing by $x^h$ and regrouping by the sets $S_i$ gives \eqref{eq:chromatic-polymer}.
\end{proof}

Define the local activity norm
\[
\delta_x:=
\max_{v\in V(H)}
\sum_{\substack{S\subseteq V(H)\\v\in S,\ |S|\ge2}}
|z_x(S)|e^{|S|}.
\]

\begin{lemma}\label{lem:chromatic-convergence}
If $x>4e\Delta$, then
\begin{equation}\label{eq:chromatic-delta}
\delta_x
\le e\sum_{k\ge1}\left(\frac{4e\Delta}{x}\right)^k
=\frac{4e^2\Delta/x}{1-4e\Delta/x}.
\end{equation}
In particular, $\delta_x\le1/4$ whenever $x\ge C_{\mathrm{chr}}\Delta$ for a sufficiently large absolute constant $C_{\mathrm{chr}}$.
\end{lemma}

\begin{proof}
For the contribution of the sets of size $s$ to $\delta_x$, Lemma~\ref{lem:tree-counts}(1) gives
\[
(4\Delta)^{s-1}x^{1-s}e^s
=e\left(\frac{4e\Delta}{x}\right)^{s-1}.
\]
Summing over all integers $s \ge 2$ yields \eqref{eq:chromatic-delta}.
\end{proof}

\begin{proof}[Proof of Theorem~\ref{thm:chromatic}]
Use polymers $S\subseteq V(H)$, $|S|\ge2$, with activities $z_x(S)$ and compatibility given by disjointness. Thus a compatible configuration is a family of pairwise vertex-disjoint source sets, and the activity of $S$ is the signed total contribution of all connected spanning edge sets of $H[S]$, multiplied by $x^{1-|S|}$. Lemma~\ref{lem:chromatic-polymer} says that the resulting partition function is exactly $P_H(x)/x^h$.

If $x\ge C_{\mathrm{chr}}\Delta$, then Lemma~\ref{lem:chromatic-convergence} gives $\delta_x\le1/4$, so Lemma~\ref{lem:KP} applies. Thus the partition function is nonzero; being real, continuous, and tending to $1$ as $x\to\infty$, it is positive and its real logarithm is defined.

Applying Lemma~\ref{lem:differential-polymer} gives
\begin{equation}\label{eq:chromatic-differential}
\left|
\frac{\dd}{\dd x}\log\!\left(\frac{P_H(x)}{x^h}\right)
-\sum_{|S|\ge2}z'_x(S)
\right|
\le
\sum_{|S|\ge2}|z'_x(S)|
\bigl(e^{2\delta_x|S|}-1\bigr).
\end{equation}
We estimate separately the two-vertex polymers and the polymers with at least three vertices.

If $|S|=2$, then $\chi_H(S)=-1$ when $S$ is an edge and is zero otherwise. Hence $z_x(S)=-x^{-1}$ for each of the $m$ edges, and therefore
\begin{equation}\label{eq:chromatic-two-vertex}
\sum_{|S|=2}z'_x(S)=\frac{m}{x^2}.
\end{equation}
For $|S|=s$,
\[
|z'_x(S)|=(s-1)|\chi_H(S)|x^{-s}.
\]
Writing $k=s-2$ and applying Lemma~\ref{lem:tree-counts}(2), we obtain
\begin{align}
\sum_{|S|\ge3}|z'_x(S)|
&\le\frac{8m}{x^2}\sum_{k\ge1}(k+1)
\left(\frac{4\Delta}{x}\right)^k,\label{eq:chromatic-high-plain}\\
\sum_{|S|\ge3}|z'_x(S)|e^{|S|}
&\le\frac{8me^2}{x^2}\sum_{k\ge1}(k+1)
\left(\frac{4e\Delta}{x}\right)^k.\label{eq:chromatic-high-weighted}
\end{align}
Both right-hand sides are $O(m\Delta/x^3)$ when $x\ge C_{\mathrm{chr}}\Delta$.

The two-vertex contribution to the right-hand side of \eqref{eq:chromatic-differential} is
\[
\frac{m}{x^2}(e^{4\delta_x}-1)
=O\!\left(\frac{m\Delta}{x^3}\right),
\]
since \eqref{eq:chromatic-delta} gives $\delta_x=O(\Delta/x)$. For $|S|\ge3$, the condition $\delta_x\le1/4$ shows that
\[
e^{2\delta_x|S|}-1\le e^{|S|},
\]
so \eqref{eq:chromatic-high-weighted} bounds the remaining correction by $O(m\Delta/x^3)$. Combining these estimates with \eqref{eq:chromatic-two-vertex} and \eqref{eq:chromatic-high-plain}, there is an absolute constant $K>0$ such that
\[
\left|
\frac{\dd}{\dd x}\log\!\left(\frac{P_H(x)}{x^h}\right)-\frac{m}{x^2}
\right|
\le K\frac{m\Delta}{x^3}.
\]
Enlarge the absolute constant $C$ so that $C\ge100K$. For $x\ge C\Delta$, the error is at most $m/(2x^2)$. The logarithmic derivative is positive, and the positivity of $P_H(x)/x^h$ implies strict monotonicity.
\end{proof}

\begin{corollary}\label{cor:adjacent-colors}
There exists an absolute constant $C>0$ such that, if $r\ge2$ and $r\ge C\Delta(J)$, then
\begin{equation}\label{eq:adjacent-colors}
\left|
\log\frac{p_{r-1}(J)}{p_r(J)}
+\frac{e(J)}{r(r-1)}
\right|
\le C\frac{e(J)\Delta(J)}{r^3}.
\end{equation}
\end{corollary}

\begin{proof}
The assertion is trivial when $e(J)=0$. Otherwise, after enlarging $C$, Theorem~\ref{thm:chromatic} applies throughout $[r-1,r]$. Integrating its logarithmic derivative gives
\[
\log\frac{p_r(J)}{p_{r-1}(J)}
=\int_{r-1}^r
\frac{\dd}{\dd x}\log\!\left(\frac{P_J(x)}{x^{v(J)}}\right)\dd x
=\frac{e(J)}{r(r-1)}
+O\!\left(\frac{e(J)\Delta(J)}{r^3}\right),
\]
which is equivalent to \eqref{eq:adjacent-colors}.
\end{proof}

\section{Proof of Theorem~\ref{thm:main}}\label{sec:main-proof}

\subsection{The minimum-degree bound}

Recall that $h=v(H)$ and $m=e(H)$, and let $C_T$ be an absolute constant to be fixed below. We prove Theorem~\ref{thm:main} by induction on $h$ in the following slightly stronger form. For every graph $J$ with $v(J)\le h$, every integer $r\ge C_Th$, and all sufficiently large $n$, every $n$-vertex $K_{r+1}$-free graph $F$ satisfies
\[
\inj(J,F)\le\inj(J,T_r(n)).
\]
When $e(J)>0$, equality holds only for $F\cong T_r(n)$. The case $h\le2$ follows from Tur\'an's theorem.

By \eqref{eq:isolated}, isolated vertices may be removed. We may therefore assume that $H$ has order $h\ge3$, contains no isolated vertices, and has $m>0$ edges. Fix $r\ge C_Th$, assume the assertion for graphs of order less than $h$, and let $G$ maximize $\inj(H,G)$ among all $n$-vertex $K_{r+1}$-free graphs.

By Corollary~\ref{cor:color-deletion}, if $J$ has at most $h$ vertices and $D\subseteq E(J)$, then
\[
\log\frac{p_r(J-D)}{p_r(J)}
\le |D|\left(\frac1r+C\frac{h}{r^2}\right).
\]
By Corollary~\ref{cor:adjacent-colors}, if $J$ has at most $h$ vertices, then
\[
\left|
\log\frac{p_{r-1}(J)}{p_r(J)}
+\frac{e(J)}{r(r-1)}
\right|
\le C\frac{e(J)h}{r^3}.
\]
We shall also use Corollaries~\ref{cor:finite-switching} and \ref{cor:discrete-balancing} in Subsection~\ref{subsec:completion}.

For $v\in V(G)$, let $I_G(v)$ denote the number of injective homomorphisms $H\to G$ whose image contains $v$.

\begin{lemma}\label{lem:cloning}
Uniformly in $v\in V(G)$,
\begin{equation}\label{eq:cloning}
I_G(v)\ge h p_r(H)n^{h-1}-O_{H,r}(n^{h-2}).
\end{equation}
\end{lemma}

\begin{proof}
Every injective copy is counted at each of its $h$ image vertices, so
\[
\sum_{v\in V(G)}I_G(v)=h\inj(H,G).
\]
Choose $x\in V(G)$ with maximum $I_G(x)$. For a fixed $v\ne x$, replace $v$ by a copy of $x$: delete every edge incident with $v$, set
\[
N(v)=N_G(x)\setminus\{v\},
\]
and keep $xv$ absent. The resulting graph is still $K_{r+1}$-free, since a clique containing the new vertex $v$ can be transformed into a clique of the same order by replacing $v$ with $x$.

Every copy using $x$ but not $v$ transfers to a copy using the new vertex $v$. The only possible discrepancy comes from injective maps whose images contain both $x$ and $v$. By \eqref{eq:collision}, or directly by choosing the ordered preimages of $x,v$ and then the remaining images, their number is at most
\[
h(h-1)(n-2)_{h-2}=O_H(n^{h-2}).
\]
Extremality of $G$ therefore gives
\begin{equation}\label{eq:cloning-intermediate}
I_G(v)
\ge I_G(x)-O_H(n^{h-2})
\ge\frac{h}{n}\inj(H,G)-O_H(n^{h-2}).
\end{equation}
The Tur\'an graph is admissible, so $\inj(H,G)\ge\inj(H,T_r(n))$. If its part proportions are $q_i=|A_i|/n$, then $q_i=1/r+O_r(n^{-1})$. Since $\Phi_H$ is a fixed polynomial,
\[
\homc(H,T_r(n))
=n^h\Phi_H(q_1,\ldots,q_r)
=p_r(H)n^h+O_{H,r}(n^{h-1}).
\]
Applying \eqref{eq:collision} with $J=H$ and $F=T_r(n)$ gives
\begin{equation}\label{eq:turan-inj-asymptotic}
\inj(H,T_r(n))=p_r(H)n^h+O_{H,r}(n^{h-1}).
\end{equation}
Substituting \eqref{eq:turan-inj-asymptotic} into \eqref{eq:cloning-intermediate} yields \eqref{eq:cloning}.
\end{proof}

We now derive an upper bound on the same rooted count. Fix $v\in V(G)$ and define
\begin{equation}\label{eq:A-B-eta}
A=N_G(v),
\qquad
B=V(G)\setminus(A\cup\{v\}),
\qquad
\eta=\frac{n-d_G(v)}{n}.
\end{equation}
Then
\begin{equation}\label{eq:A-B-sizes}
|A|=(1-\eta)n,
\qquad
|B|=\eta n-1.
\end{equation}
The graph $G[A]$ is $K_r$-free and $G[B]$ is $K_{r+1}$-free.

For $u\in V(H)$, let $I_{G,u}(v)$ count the injective homomorphisms $\varphi:H\to G$ satisfying $\varphi(u)=v$. Set
\[
\overline N_H(u)=V(H)\setminus(N_H(u)\cup\{u\}).
\]
Every map counted by $I_{G,u}(v)$ determines a partition
\[
V(H)\setminus\{u\}=S\dotcup T,
\]
where $S$ is the set of vertices of $H$ embedded in $B$ and $T$ is the set of vertices of $H$ embedded in $A$. Since every neighbor of $u$ must be embedded in $A$, necessarily $S\subseteq\overline N_H(u)$ and
\[
T=V(H)\setminus(S\cup\{u\}).
\]
For a fixed $S$, the restrictions of $\varphi$ to $S$ and $T$ have disjoint ranges because $A\cap B=\varnothing$. If we ignore the edge conditions between $S$ and $T$, we obtain
\begin{equation}\label{eq:rooted-partition-upper}
I_{G,u}(v)
\le
\sum_{S\subseteq\overline N_H(u)}
\inj(H[S],G[B])\inj(H[T],G[A]).
\end{equation}

The induction hypothesis has the following uniform consequence. If $J$ has fewer than $h$ vertices and $q\ge C_Tv(J)$, then
\begin{equation}\label{eq:uniform-induction}
\inj(J,F)
\le p_q(J)N^{v(J)}+C_{J,q}(N+1)^{v(J)-1}
\end{equation}
for every $N$-vertex $K_{q+1}$-free graph $F$. For the finitely many host orders below the induction threshold, the constant $C_{J,q}$ is enlarged. The empty graph contributes $1$ and needs no error term.

We apply \eqref{eq:uniform-induction} to $G[B]$ with $q=r$ and to $G[A]$ with $q=r-1$. This is legitimate because $r-1\ge C_T(h-1)\ge C_Tv(J)$ after taking $C_T\ge1$. Equations \eqref{eq:rooted-partition-upper} and \eqref{eq:A-B-sizes} give
\begin{equation}\label{eq:rooted-Ru}
I_{G,u}(v)\le n^{h-1}R_u(\eta)+O_{H,r}(n^{h-2}),
\end{equation}
where
\begin{equation}\label{eq:Ru-definition}
R_u(\eta)=
\sum_{S\subseteq\overline N_H(u)}
\eta^{|S|}(1-\eta)^{h-1-|S|}
 p_r(H[S])p_{r-1}(H[T]).
\end{equation}

\begin{lemma}\label{lem:rooted-coloring}
There exists an absolute constant $C$ such that, if $r\ge Ch$, then, for $0\le\eta<1$,
\begin{equation}\label{eq:rooted-coloring}
\log\frac{R_u(\eta)}{p_r(H)}
\le
 d_H(u)\left(\log(1-\eta)+\frac1r+C\frac{h}{r^2}\right)
+C\frac{\eta m}{r}+C\frac{mh}{r^3}.
\end{equation}
For $\eta=1$, the assertion is interpreted as the limit from below.
\end{lemma}

\begin{proof}
Fix $S\subseteq\overline N_H(u)$ and set
\[
T=V(H)\setminus(S\cup\{u\}),
\qquad
d=d_H(u),
\qquad
c(S)=e_H(S,T).
\]
Let $D$ consist of the $d$ edges incident with $u$ and the $c(S)$ edges joining $S$ to $T$. Then
\[
H-D=H[S]\dotcup H[T]\dotcup K_1.
\]
The proper-coloring probabilities multiply over disjoint unions. Hence Corollary~\ref{cor:color-deletion} gives the complete chain
\begin{equation}\label{eq:rooted-delete-chain}
p_r(H[S])p_r(H[T])
=p_r(H-D)
\le p_r(H)\exp\!\left((d+c(S))\left(\frac1r+C\frac{h}{r^2}\right)\right).
\end{equation}
According to Corollary~\ref{cor:adjacent-colors}, we have
\begin{equation}\label{eq:rooted-adjacent}
\log\frac{p_{r-1}(H[T])}{p_r(H[T])}
\le-\frac{e(H[T])}{r(r-1)}+C\frac{e(H[T])h}{r^3}
\le C\frac{mh}{r^3}.
\end{equation}
Define
\[
\alpha=\frac1r+C\frac{h}{r^2}.
\]
Combining \eqref{eq:rooted-delete-chain} and \eqref{eq:rooted-adjacent}, we obtain
\begin{equation}\label{eq:rooted-product-bound}
p_r(H[S])p_{r-1}(H[T])
\le p_r(H)\exp\!\left(\alpha(d+c(S))+C\frac{mh}{r^3}\right).
\end{equation}

Let $\mathbf S$ be a random subset of $\overline N_H(u)$ obtained by independently placing each $x\in\overline N_H(u)$ in $\mathbf S$ with probability $\eta$. Write
\[
X_x=\1_{\{x\in\mathbf S\}}.
\]
The definition \eqref{eq:Ru-definition} may then be written as
\[
R_u(\eta)
=(1-\eta)^d
\E\!\bigl[p_r(H[\mathbf S])p_{r-1}(H[V(H)\setminus(\mathbf S\cup\{u\})])\bigr].
\]
Indeed, all $d$ neighbors of $u$ belong to $T$, and the remaining $h-1-d$ vertices are independently classified into $S$ and $T$. By \eqref{eq:rooted-product-bound},
\begin{equation}\label{eq:Ru-expectation}
\frac{R_u(\eta)}{p_r(H)}
\le(1-\eta)^d
\exp\!\left(\alpha d+C\frac{mh}{r^3}\right)
\E e^{\alpha c(\mathbf S)}.
\end{equation}
Every edge counted by $c(\mathbf S)$ has one endpoint $x\in\mathbf S$, so
\[
c(\mathbf S)\le\sum_{x\in\overline N_H(u)}d_H(x)X_x.
\]
Independence gives
\[
\E e^{\alpha c(\mathbf S)}
\le\prod_{x\in\overline N_H(u)}
\bigl(1-\eta+\eta e^{\alpha d_H(x)}\bigr),
\]
and hence
\[
\log\E e^{\alpha c(\mathbf S)}
\le\sum_{x\in\overline N_H(u)}
\log\bigl(1+\eta(e^{\alpha d_H(x)}-1)\bigr).
\]
Since $\alpha h=O(h/r)$ is small, $e^{\alpha d_H(x)}-1\le C\alpha d_H(x)$. Therefore
\begin{equation}\label{eq:mgf-bound}
\log\E e^{\alpha c(\mathbf S)}
\le C\eta\alpha\sum_{x\in V(H)}d_H(x)
\le C\frac{\eta m}{r}.
\end{equation}
Taking logarithms in \eqref{eq:Ru-expectation} and using \eqref{eq:mgf-bound} yields \eqref{eq:rooted-coloring}. If $\eta=1$, then $R_u(1)=0$ because $d_H(u)>0$, which agrees with the limiting interpretation.
\end{proof}

\begin{lemma}\label{lem:min-degree}
If the absolute constant $C_T$ is sufficiently large, then every extremal graph $G$ above satisfies
\begin{equation}\label{eq:min-degree}
\delta(G)>\left(1-\frac8r\right)n.
\end{equation}
Equivalently,
\begin{equation}\label{eq:rho-extremal}
\rho(G)<\frac8r.
\end{equation}
\end{lemma}

\begin{proof}
Choose a sufficiently large absolute constant $K$. Next choose $C_T$, depending only on $K$ and the absolute constants in the preceding estimates, so that all absorptions below hold. Finally, take $n$ sufficiently large in terms of the fixed $H$ and $r$. Partition the vertices of $H$ into
\[
V_{\mathrm{lo}}
=\left\{u:d_H(u)<\frac{Km}{r}\right\},
\qquad
V_{\mathrm{hi}}=V(H)\setminus V_{\mathrm{lo}}.
\]
Write
\[
D_{\mathrm{lo}}=\sum_{u\in V_{\mathrm{lo}}}d_H(u),
\qquad
D_{\mathrm{hi}}=\sum_{u\in V_{\mathrm{hi}}}d_H(u).
\]
By the choice of $C_T$,
\begin{equation}\label{eq:D-low-high}
D_{\mathrm{lo}}\le\frac{Khm}{r}\le\frac m2,
\qquad
D_{\mathrm{hi}}=2m-D_{\mathrm{lo}}\ge\frac{3m}{2}.
\end{equation}

Fix $v\in V(G)$ and let $\eta$ be as in \eqref{eq:A-B-eta}. We shall sum the numbers of maps according to the vertex of $H$ sent to $v$ and according to its degree class:
\begin{equation}\label{eq:rooted-class-sum}
I_G(v)
=\sum_{u\in V(H)}I_{G,u}(v)
=\sum_{u\in V_{\mathrm{hi}}}I_{G,u}(v)
+\sum_{u\in V_{\mathrm{lo}}}I_{G,u}(v).
\end{equation}
If $\eta=1$, then $v$ is isolated in $G$. Since $H$ has no isolated vertices, no injective copy of $H$ can contain $v$, so $I_G(v)=0$, contradicting Lemma~\ref{lem:cloning} for sufficiently large $n$. Hence $\eta<1$. Suppose for a contradiction that $\eta\ge8/r$.

For $u\in V_{\mathrm{hi}}$, Lemma~\ref{lem:rooted-coloring}, the inequality $\log(1-\eta)\le-\eta$, and the choices of $K,C_T$ give
\begin{equation}\label{eq:Ru-high}
R_u(\eta)\le p_r(H)e^{-\eta d_H(u)/2}.
\end{equation}
Indeed, $d_H(u)/r\le\eta d_H(u)/8$. Since $d_H(u)\ge Km/r$, the term $C\eta m/r$ is at most $(C/K)\eta d_H(u)$ and is absorbed by first choosing $K$ large. The terms $Chd_H(u)/r^2$ and $Cmh/r^3$ are then absorbed by choosing $C_T$ large, using $r\ge C_Th$ and $\eta\ge8/r$. Thus the whole exponent is at most $-\eta d_H(u)/2$. Summing \eqref{eq:rooted-Ru} and \eqref{eq:Ru-high} over $V_{\mathrm{hi}}$ gives
\begin{equation}\label{eq:high-rooted-sum}
\sum_{u\in V_{\mathrm{hi}}}I_{G,u}(v)
\le p_r(H)n^{h-1}
\sum_{u\in V_{\mathrm{hi}}}e^{-\eta d_H(u)/2}
+O_{H,r}(n^{h-2}).
\end{equation}
Since $1-e^{-x/2}\ge x/(2+x)$ for $x\ge0$,
\begin{align}
\sum_{u\in V_{\mathrm{hi}}}e^{-\eta d_H(u)/2}
&=|V_{\mathrm{hi}}|
-\sum_{u\in V_{\mathrm{hi}}}\bigl(1-e^{-\eta d_H(u)/2}\bigr)\notag\\
&\le |V_{\mathrm{hi}}|
-\sum_{u\in V_{\mathrm{hi}}}\frac{\eta d_H(u)}{2+\eta d_H(u)}\notag\\
&\le |V_{\mathrm{hi}}|-\frac{\eta D_{\mathrm{hi}}}{2+\eta h}.
\label{eq:high-exponential-sum}
\end{align}

For $u\in V_{\mathrm{lo}}$, forget every edge condition incident with the root. The induction hypothesis applied to $H-u$ gives
\[
I_{G,u}(v)
\le\inj(H-u,G-v)
\le p_r(H-u)n^{h-1}+O_{H,r}(n^{h-2}).
\]
Deleting the $d_H(u)$ edges incident with $u$ and applying Corollary~\ref{cor:color-deletion} yields
\[
p_r(H-u)\le p_r(H)e^{Cd_H(u)/r}.
\]
Since $d_H(u)\le Km/r$ and $r\ge C_Th$, the exponent is uniformly small. Hence
\begin{align}
\sum_{u\in V_{\mathrm{lo}}}I_{G,u}(v)
&\le p_r(H)n^{h-1}
\sum_{u\in V_{\mathrm{lo}}}e^{Cd_H(u)/r}
+O_{H,r}(n^{h-2})\notag\\
&\le p_r(H)n^{h-1}
\left(|V_{\mathrm{lo}}|+C\frac{D_{\mathrm{lo}}}{r}\right)
+O_{H,r}(n^{h-2}).
\label{eq:low-rooted-sum}
\end{align}
Combining \eqref{eq:rooted-class-sum}, \eqref{eq:high-rooted-sum}, \eqref{eq:high-exponential-sum} and \eqref{eq:low-rooted-sum}, we obtain
\[
I_G(v)
\le p_r(H)n^{h-1}
\left(h+C\frac{D_{\mathrm{lo}}}{r}
-\frac{\eta D_{\mathrm{hi}}}{2+\eta h}\right)
+O_{H,r}(n^{h-2}).
\]
By \eqref{eq:D-low-high}, $\eta\ge8/r$, and $r\ge8h$, we obtain
\begin{equation}\label{eq:degree-gap}
C\frac{D_{\mathrm{lo}}}{r}\le C\frac{Khm}{r^2},
\qquad
\frac{\eta D_{\mathrm{hi}}}{2+\eta h}\ge c\frac mr.
\end{equation}
More explicitly, $D_{\mathrm{lo}}\le Khm/r$ makes the first term at most $CKhm/r^2$. Also $D_{\mathrm{hi}}\ge3m/2$, while $\eta\ge8/r$ and $r\ge8h$ imply
\[
\frac{\eta D_{\mathrm{hi}}}{2+\eta h}\ge\frac{4m}{r}.
\]
The choice of $C_T$ therefore makes the first term in \eqref{eq:degree-gap} at most half the second. Thus, for some constant $c_{H,r}>0$,
\[
I_G(v)
\le hp_r(H)n^{h-1}-c_{H,r}n^{h-1}+O_{H,r}(n^{h-2}),
\]
contradicting Lemma~\ref{lem:cloning} for sufficiently large $n$. Therefore $\eta<8/r$ for every $v$, which proves \eqref{eq:min-degree} and \eqref{eq:rho-extremal}.
\end{proof}

\subsection{F\"uredi reduction and completion}\label{subsec:completion}

We use the following exact stability form of Tur\'an's theorem.

\begin{theorem}[F\"uredi \cite{Furedi}]\label{thm:Furedi}
Every $n$-vertex $K_{r+1}$-free graph $F$ has a spanning, at most $r$-partite subgraph $F_0$ such that
\[
e(F_0)\ge2e(F)-e(T_r(n)).
\]
\end{theorem}

F\"uredi states the result in the equivalent form that, if $t=e(T_r(n))-e(F)$, then one can retain at least $e(F)-t$ edges in an at most $r$-partite subgraph. Vertices omitted from the displayed parts may be added as isolated vertices.

\begin{proof}[Proof of Theorem~\ref{thm:main}]
Choose a partition
\[
\mathcal P=(A_1,\ldots,A_r)
\]
of $V(G)$ that maximizes the number of crossing edges of $G$. Let $G_0$ consist of these crossing edges, and let $K_{\mathcal P}$ be the complete $r$-partite graph with the same classes. Define
\[
L=e(G)-e(G_0),
\qquad
M=e(K_{\mathcal P})-e(G_0),
\qquad
B=e(T_r(n))-e(K_{\mathcal P}).
\]
Thus $L$ is the number of deleted internal edges, $M$ is the number of missing crossing edges, and $B$ is the loss caused by unbalanced class sizes.

By maximality of the cut, $e(G_0)$ is at least the number of edges in the spanning at most $r$-partite subgraph supplied by Theorem~\ref{thm:Furedi}. Therefore
\begin{equation}\label{eq:maxcut-Furedi}
e(G_0)\ge2e(G)-e(T_r(n)).
\end{equation}
Substituting
\[
e(G)=e(G_0)+L,
\qquad
e(T_r(n))=e(G_0)+M+B
\]
into \eqref{eq:maxcut-Furedi} and cancelling $e(G_0)$ gives
\begin{equation}\label{eq:Furedi-assembly}
2L\le B+M.
\end{equation}

We next verify that all intermediate graphs are dense enough for the switching estimates. If $v\in A_i$, moving $v$ from $A_i$ to $A_j$ cannot increase the number of crossing edges. Hence
\begin{equation}\label{eq:maxcut-vertex}
e_G(v,A_i)\le e_G(v,A_j)
\qquad(j\ne i).
\end{equation}
Summing \eqref{eq:maxcut-vertex} over $j\ne i$ gives
\[
(r-1)e_G(v,A_i)\le d_G(v)-e_G(v,A_i),
\qquad
e_G(v,A_i)\le\frac{d_G(v)}{r}<\frac nr.
\]
By Lemma~\ref{lem:min-degree}, deleting all internal edges increases the normalized missing degree by at most $1/r$, while adding crossing edges cannot increase it. Thus every graph encountered while passing from $G$ to $G_0$ by deletions and then from $G_0$ to $K_{\mathcal P}$ by additions satisfies
\begin{equation}\label{eq:rho-intermediate}
\rho<\frac9r.
\end{equation}
Moreover, if $v\in A_i$, every vertex of $A_i$, including $v$, is a nonneighbor of $v$ in $G_0$. Consequently,
\begin{equation}\label{eq:part-flatness}
\frac{|A_i|}{n}\le\rho(G_0)<\frac9r,
\qquad
\max_i|A_i|\le\frac{10n}{r}
\end{equation}
for sufficiently large $n$.

Choose $C_T$ sufficiently large, and then choose $n$ sufficiently large, so that all switching and balancing estimates apply. Set
\[
\varepsilon_n=C\frac{h}{r}+O_{H,r}(n^{-1})<\frac14,
\]
where $C$ dominates the constants in Corollaries~\ref{cor:finite-switching} and \ref{cor:discrete-balancing}.

Before taking logarithms, we verify that every count below is positive. The inequality $h\rho<1$ permits a greedy injective embedding of $H$ into each intermediate dense graph: when a vertex of $H$ is embedded, fewer than $h$ already used vertices and at most $h\rho n$ forbidden nonneighbors have to be avoided. At the complete multipartite endpoints, positivity follows from $r\ge h$ by assigning distinct parts to the vertices of $H$.

Delete the $L$ internal edges one at a time. Corollary~\ref{cor:finite-switching}(2), together with \eqref{eq:rho-intermediate}, gives
\begin{equation}\label{eq:delete-internal}
\log\frac{\inj(H,G)}{\inj(H,G_0)}
\le\frac{2m}{n^2}(1+\varepsilon_n)L.
\end{equation}
Next add the $M$ missing crossing edges one at a time. Corollary~\ref{cor:finite-switching}(1), again together with \eqref{eq:rho-intermediate}, gives
\begin{equation}\label{eq:add-crossing}
\log\frac{\inj(H,K_{\mathcal P})}{\inj(H,G_0)}
\ge\frac{2m}{n^2}(1-\varepsilon_n)M.
\end{equation}
Finally, \eqref{eq:part-flatness} verifies the flatness hypothesis of Corollary~\ref{cor:discrete-balancing}. That corollary gives
\begin{equation}\label{eq:balance-final}
\log\frac{\inj(H,T_r(n))}{\inj(H,K_{\mathcal P})}
\ge\frac{2m}{n^2}(1-\varepsilon_n)B.
\end{equation}
Adding \eqref{eq:add-crossing} and \eqref{eq:balance-final}, and then subtracting \eqref{eq:delete-internal}, yields
\begin{align}
\log\frac{\inj(H,T_r(n))}{\inj(H,G)}
&\ge\frac{2m}{n^2}
\bigl((1-\varepsilon_n)(B+M)-(1+\varepsilon_n)L\bigr)\notag\\
&\ge\frac{m(B+M)}{n^2}(1-3\varepsilon_n).
\label{eq:final-comparison}
\end{align}
The second inequality follows from \eqref{eq:Furedi-assembly}:
\begin{align*}
(1-\varepsilon_n)(B+M)-(1+\varepsilon_n)L
&\ge(1-\varepsilon_n)(B+M)
-\frac{1+\varepsilon_n}{2}(B+M)\\
&=\frac{1-3\varepsilon_n}{2}(B+M).
\end{align*}
If $B+M>0$, then \eqref{eq:final-comparison} is positive, contrary to the extremality of $G$. Hence $B=M=0$. Equation \eqref{eq:Furedi-assembly} then gives $L=0$, so $G=K_{\mathcal P}$. Since $B=0$, the part sizes differ by at most one. Indeed, moving one vertex from a part that exceeds another by at least two would increase the number of edges. Therefore $G\cong T_r(n)$. This completes the induction and proves Theorem~\ref{thm:main}.
\end{proof}

\section{Conclusioning remarks}
In this paper, we have used the cluster method to prove the vertex-linear Tur\'an-goodness threshold and the chromatic-polynomial monotonicity. The cluster method is a powerful tool for studying the structure of graphs and their properties. 

  In Theorem~\ref{thm:main}, the stability statement also holds. More precisely, if
  $e(H)>0$ and $r\ge Cv(H)$, then, for every $\varepsilon>0$, there exist
  $\delta>0$ and $n_0$ such that every $n\ge n_0$ and every $n$-vertex
  $K_{r+1}$-free graph $G$ satisfying
  \[
  N(H,G)\ge N(H,T_r(n))-\delta n^{v(H)}
  \]
  can be transformed into $T_r(n)$ by adding and deleting at most
  $\varepsilon n^2$ edges. Thus $H$ is $K_{r+1}$-Tur\'an-stable.
  
  Indeed, the cloning argument used in the proof of
  Lemma~\ref{lem:min-degree} eliminates all low-degree vertices after
  $o(n)$ symmetrization steps, and the switching and balancing argument
  leading to \eqref{eq:final-comparison} then gives the required
  $o(n^2)$ edit bound, compare \cite{PartI}. By the standard stability
  transfer for generalized Tur\'an problems, the same conclusion holds
  with $K_{r+1}$ replaced by any fixed graph $F$ with
  $\chi(F)=r+1$. Moreover, if $F$ has a color-critical edge, the standard
  exactness reduction to complete $r$-partite graphs, together with
  Theorem~\ref{thm:main}, gives
  \[
  \operatorname{ex}(n,H,F)=N(H,T_r(n))
  \]
  for all sufficiently large $n$. In particular, $H$ is
  $F$-Tur\'an-good in the same vertex-linear range. We omit the routine
  details in order to keep the presentation concise.

\section*{Acknowledgments}

The authors also acknowledge the use of OpenAI's ChatGPT 5.6 as an interactive tool for exploring proof strategies, checking calculations, and improving exposition.
More specifically, it helped to prove Lemma \ref{lem:polymer-perturbation}, which brings the cluster method to the present problem.
The authors take full responsibility for every statement, proof, citation, and conclusion in the paper.

\end{document}